\documentclass[12pt]{amsart}
\usepackage{amssymb}
\usepackage{amsxtra}
\usepackage[all]{xy}
\usepackage{graphics}

\newtheorem{thm}{Theorem}[section]
\newtheorem{lem}[thm]{Lemma}

\newtheorem{prop}[thm]{Proposition}

\newtheorem*{prob*}{Open problem}

\theoremstyle{definition}
\newtheorem{defi}[thm]{Definition}

\theoremstyle{remark}
\newtheorem*{rem*}{Remark}

\DeclareMathOperator{\id}{id}

\newcommand{\kringel}{\mathbin{\raise1pt\hbox{$\scriptstyle\circ$}}} 
\newcommand{\pkt}{\mathbin{\raise0pt\hbox{$\scriptstyle\bullet$}}}

\newcommand{\C}{\mathbb{C}}

\newcommand{\N}{\mathbb{N}}

\newcommand{\tr}{\mathop{\rm tr}}

\newcommand{\End}{\mathop{\rm End}}
\newcommand{\Der}{\mathop{\rm Der}}
\newcommand{\diag}{\mathop{\rm diag}}

\newcommand{\Lg}{\mathfrak{g}}

\newcommand{\Ln}{\mathfrak{n}}

\newcommand{\Lr}{\mathfrak{r}}

\newcommand{\CL}{\mathcal{L}}

\newcommand{\CR}{\mathcal{R}}

\newcommand{\al}{\alpha}
\newcommand{\be}{\beta}
\newcommand{\ga}{\gamma}

\newcommand{\la}{\lambda}

\newcommand{\ov}{\overline}

\newcommand{\ra}{\rightarrow}

\renewcommand{\phi}{\varphi}

\DeclareFontFamily{OT1}{pzc}{}
\DeclareFontShape{OT1}{pzc}{m}{it}%
              {<-> s * [1.100] pzcmi7t}{}
\DeclareMathAlphabet{\mathpzc}{OT1}{pzc}%
                                 {m}{it}

\begin{document}

\title[Classification of orbit closures]{Classification of orbit closures in the variety
of 3-dimensional Novikov algebras}
%  Die Kurzfassung kommt oben ueber die Seiten, sie steht in eckigen Klammern
%  Auch Autorennamen koennen eine Kurzfassung haben

\author[T. Bene\v{s}]{Thomas Bene\v{s}}
\author[D. Burde]{Dietrich Burde}

\address{Fakult\"at f\"ur Mathematik\\
Universit\"at Wien\\
  Nordbergstr. 15\\
  1090 Wien \\
  Austria} 
\email{thomas.benes@univie.ac.at}
\email{dietrich.burde@univie.ac.at}
\date{\today}

\subjclass[2010]{Primary 17D25, 14Q30; Secondary 81R05}
\thanks{The authors were supported by the FWF, Projekt P21683.}

\begin{abstract}
We classify the orbit closures in the variety $\mathpzc {Nov}_3$ of complex,
$3$-dimensional Novikov algebras and obtain the Hasse diagrams for the closure 
ordering of the orbits. We provide invariants which are easy to compute and which enable us  
to decide whether or not one Novikov algebra degenerates to another Novikov algebra. 
\end{abstract}

\maketitle

\section{Introduction}

Consider the variety of $K$-algebra structures on a given $n$-dimensional vector space
over a field $K$. Then the group  $GL_n(K)$ acts on these structures, and it is interesting
to study the orbit closures  with respect to the Zariski topology. If a structure representing
a $K$-algebra $B$ is contained in the orbit closure of a structure representing a
$K$-algebra $A$, then it is said that {\it $A$ degenerates to $B$}. This is denoted by
$A \ra_{\rm deg} B$ or $B\in \ov{O(A)}$. Degenerations have been first studied for
Lie algebras. There are many articles on this subject, see for example \cite{NEP} and the references 
cited therein. The motivation here came from physics, where Lie algebra degenerations have been 
studied for the special case of contractions, which are limiting processes between Lie algebras 
\cite{SEG},\cite{IWI}. For example, classical mechanics is a limiting case of quantum mechanics 
as $\hbar \ra 0$, described by a contraction of the Heisenberg-Weyl Lie algebra to the abelian 
Lie algebra of the same dimension. 
Degenerations have been also studied for other types of algebras, among them commutative algebras, 
associative algebras, Leibniz algebras \cite{FIA}, \cite{RAK}, pre-Lie algebras and 
Novikov algebras \cite{BU36}. The latter algebras have interesting applications in operad theory,
gemoetry and physics, see \cite{BU24}, \cite{BAL}, \cite{COK}. \\
In \cite{BU36} we have obtained a classification of orbit closures in the variety
$\mathpzc {Nov}_2$ of complex, $2$-dimensional Novikov algebras. In this article we 
extend this classification to dimension three. Unfortunately the complexity then
is much greater than in dimension two. Both authors have done the classification 
independently, using different methods. In the end, our results coincide. For the first
author the classification is part of his thesis \cite{BEN}. \\
To succeed in the classification, one needs effective invariants in order to decide 
whether or not a certain degeneration exists. This can always be 
decided in principle. A general result of Popov \cite{POP} implies that there exists 
an algorithm to decide whether or not one orbit lies in the orbit closure of another orbit. 
However this algorithm is not very efficient, and the calculations involved are not feasable, 
not even in dimension two. Therefore we need to provide elementary invariants which are 
easy to compute. Fortunately we have found very effective invariants. 
Nevertheless, for some cases we have to give additional arguments, such as special
applications of Borel's closed orbit lemma \cite{BOR}, \cite{GRH}, or studying left and
right annihilators under degeneration. \\

\section{The variety of Novikov algebras and degenerations}

Novikov algebras are a special class of pre-Lie algebras. For more background on these
algebras see  \cite{BU24} and the references given therein.
The definition is as follows:

\begin{defi}\label{nov}
A $K$-algebra $A$ together with a bilinear product $(x,y)\mapsto x\cdot y$ is called
a {\it Novikov algebra}, if the identities
\begin{align}
(x\cdot y) \cdot z- x \cdot (y \cdot z) & = (y\cdot x) \cdot z- y \cdot (x \cdot z)\label{nov1} \\ 
(x\cdot y)\cdot z & = (x\cdot z)\cdot y \label{nov2}
\end{align}
hold for all $x,y,z\in A$. 
\end{defi}

Denote by $L(x)$ the left multiplications, and by $R(x)$ the right multiplications.
A Novikov algebra $A$ is called {\it complete}, if all linear operators $R(x)$ are nilpotent.
The commutator $[x,y]=x\cdot y-y\cdot x$ defines a Lie bracket. We denote the associated 
Lie algebra by $\Lg_A$. Let $V$ be a vector space of dimension $n$ over $K$. 
Fix a basis $(e_1,\ldots,e_n)$ of $V$. If $(x,y)\mapsto x\cdot y$ is a Novikov algebra product 
on $V$ with $e_i\cdot e_j= \sum_{k=1}^n c_{ij}^k e_k$, then $(c_{ij}^k)\in K^{n^3}$
is called a {\it Novikov algebra structure} on $V$. These structures are points
of the {\it variety} of Novikov algebra structures. Here variety means that this is an affine 
algebraic set. Indeed, the Novikov identities are given by polynomials in the structure 
constants $c_{ij}^k$. 

\begin{defi}
Denote by $\mathpzc {Nov}_n(K)$ the set of all Novikov algebra structures on an
$n$-dimensional vector space $V$ over $K$. This set is called the 
{\it variety of Novikov algebra structures}.
\end{defi}

The general linear group $GL_n(K)$ acts on $\mathpzc {Nov}_n(K)$ by
\begin{equation*}
(g\kringel \mu)(x,y)=g(\mu(g^{-1}x, g^{-1}y))
\end{equation*}
for $g\in GL_n(K)$ and $x,y\in V$. Denote by $O(\mu)$ the orbit of $\mu$ under this action, 
and by $\ov{O(\mu)}$ the closure of the orbit with respect to the Zariski topology. 
Recall the following notations from \cite{BU36}:

\begin{defi}
Let $A$ and $B$ be two $n$-dimensional Novikov algebras. The algebra $A$ {\it degenerates}
to the algebra $B$, if the algebra $B$ is represented by a structure $\mu \in \mathpzc {Nov}_n(K)$ which
lies in the Zariski closure $\ov{O(\la)}$ of the $GL_n(K)$-orbit of some structure 
$\la\in \mathpzc {Nov}_n(K)$ which represents $A$. We write $A\ra_{\rm deg}B$.
\end{defi}

We also use the notations $O(A)$ and $\ov{O(A)}$ for the orbit and its closure of a
structure $\la\in \mathpzc {Nov}_n(K)$ representing $A$.

\begin{defi}
Let $A$ and $B$ be two $n$-dimensional Novikov algebras. We say that a degeneration 
$A \ra_{\rm deg} B$ is {\it proper}, if $B \in \partial O(A)=\ov{O(A)}\setminus O(A)$, i.e., 
if $A$ and $B$ are not isomorphic.
\end{defi}

The process of degeneration in $\mathpzc {Nov}_n(K)$ defines a partial order on the orbit space
of $n$-dimensional Novikov algebra structures, given by $O(B)\le O(A) \iff B\in \ov{O(A)}$.
We write $O(B)< O(A) \iff B\in \partial O(A)$.
In particular this relation is transitive:  If $A \ra_{\rm deg} B$ and  $B \ra_{\rm deg} C$,
then $A \ra_{\rm deg} C$. 
\begin{defi}
A degeneration $A \ra_{\rm deg} B$ is called {\it essential}, if it cannot
be obtained by transitivity. 
\end{defi}

Some authors use a slighly more imaginative expression. In \cite{DOK}, $O(C)$ is called a 
{\it child} of $O(A)$ if $O(C) < O(A)$ and there is no algebra $B$ such that
$O(C ) < O(B) < O(A)$. \\[0.2cm]
We can represent the degenerations with respect to the partial order
in a diagram: order the Novikov algebras by the dimension of its derivation algebra, in each row the
algebras with the same dimension, on top the ones with the lowest derivation dimension. 
Recall that if $A\ra_{\rm deg} B$ is a proper degeneration of two $n$-dimensional Novikov algebras
then $\dim \Der(A)<\dim \Der (B)$, see \cite{BU36}.
Draw a directed arrow between two algebras $A$ and $B$, if $A$ degenerates to $B$. Arrows following
from transitivity may be omitted.

\begin{defi}
The diagram described above is called the {\it Hasse diagram} of degenerations
in $\mathpzc {Nov}_n(K)$. It shows the classification of orbit closures.
\end{defi}

Of course, we can also give Hasse diagrams for subclasses of $n$-dimensional Novikov algebras.
This is very reasonable to do, because the Hasse diagram for the whole variety, even in dimension
$3$, will be too complicated. We recall the following result from \cite{BU36} concerning
the construction of degenerations. 

\begin{lem}\label{2.7}
Every degeneration $A\ra_{\rm deg} B$ of complex Novikov algebras can be realized by a  
sequential contraction $\lim_{t \to 0}g_{t} \kringel \la=\mu$ with
$g_t\in GL_n(\C)$, $\la$ representing $A$, $\mu$ representing $B$.
\end{lem}

For two given Novikov algebras $A$ and $B$ we want to decide whether $A$ degenerates to $B$
or not. If the answer is yes, then we can find a matrix $g_t \in GL_n(\C(t))$ realizing such a 
degeneration. Otherwise we will give an argument to show that such a degeneration is impossible. 
This can be done by providing invariants for the whole orbit of the given algebra, and hence also 
for its Zariski closure. We list the invariants which we will need. For more details and proofs
see \cite{BU36}. Let $A$ be a Novikov algebra, and choose $i,j\in \N$. The generalized trace 
invariants $c_{i,j}(A)$ and $d_{i,j}(A)$, if they exist, are given by the equations
\begin{align*}
c_{i,j}(A)\cdot \tr (L(x)^i L(y)^j) & = \tr (L(x)^i)\cdot \tr (L(y)^j),\\
d_{i,j}(A)\cdot \tr (R(x)^i R(y)^j) & = \tr (R(x)^i)\cdot \tr (R(y)^j).
\end{align*}
They can be defined by the quotient, if the equations hold for all $x,y\in A$ and the denominator 
is different from the zero polynomial in the structure constants.

\begin{lem}\label{inv1}
Suppose that $A$ degenerates to $B$ and both $c_{i,j}(A)$ and $c_{i,j}(B)$ exist. Then it
follows $c_{i,j}(A)=c_{i,j}(B)$. The same applies for $d_{i,j}(A)$ and $d_{i,j}(B)$.
\end{lem}

By an operator identity in $L(x)$ and $R(x)$ for $A$ we mean a polynomial $T(x)\in \End (A)$ 
in $L(x)$ and $R(x)$, which is zero for all $x\in A$. 
For example, if $A$ is commutative, then the operator $T(x)=L(x)-R(x)$ 
satisfies $T(x)=0$ for all $x\in A$. 

\begin{lem}\label{inv2}
Suppose that $A$ degenerates to $B$ and $T(x)=0$ for all $x\in A$. Then it follows $T(x)=0$ for
all $x\in B$.
\end{lem} 

This lemma will be used to show the impossibility of certain degenerations. Suppose $A$ and $B$ are
two Novikov algebras in $\mathpzc {Nov}_n(K)$, and $A$ satisfies such a polynomial operator identity 
$T(x)=0$, but $B$ does not. Then $A$ cannot degenerate to $B$. Sometimes it is easy to find such
polynomials. If $A$ is complete, we may take $T(x)=R(x)^n$. Hence $A\ra_{\rm deg}B$ implies that 
$B$ is also complete in this case. It is also easy to see that zero trace or zero determinant
of such operators is preserved under degeneration. \\
In order to find interesting operator polynomials it is useful to realize that the following 
polynomial is zero for {\it all} Novikov algebras.
\begin{lem}
Let $A$ be a Novikov algebra. Then the following operator identity holds.
\begin{align*}
[R(x),[R(x),L(x)]] & = L(x)R(x)^2-2R(x)L(x)R(x)+R(x)^2L(x) \\
 & = 0.
\end{align*}
\end{lem}

\begin{proof}
From \eqref{nov1} we obtain the operator identity $[R(x),L(x)]=R(x)^2-R(x\cdot x)$. By
\eqref{nov2} we have $[R(x),R(y)]=0$. Hence it follows
$[R(x),[R(x),L(x)]] = [R(x),R(x)^2]-[R(x),R(x\cdot x)] = 0$.
\end{proof}

Given a Novikov algebra $(A,\cdot )$ we have defined the associated Lie algebra $\Lg_A$
by $[x,y]=x\cdot y-y\cdot x$. We can also associate a commutative algebra $J_A$ by
$x\circ y=x\cdot y+y\cdot x$. This algebra however may not be associative in general.
If it is, then it is again a Novikov algebra. The following result is easy to show, see \cite{BU36}:

\begin{lem}\label{inv5}
If $A \ra_{\rm deg} B$ then $\Lg_A \ra_{\rm deg} \Lg_B$ and $J_A \ra_{\rm deg} J_B$.
\end{lem}

\begin{defi}
Let  $\al,\be,\ga \in \C$ and define $\Der_{(\al,\be,\ga)}(A)$ to be the space of all 
$ D\in \End (A)$ satisfying
\[
\al D(x\cdot y)=\be D(x)\cdot y+\ga x \cdot D(y)
\]
for all $x,y \in A$.
We call the elements $D\in\Der_{(\al,\be,\ga)}(A)$ sometimes $(\al,\be,\ga)$-derivations.
\end{defi}

For $(\al,\be,\ga)=(1,1,1)$ we obtain the usual derivation algebra $\Der (A)$.

\begin{lem}\label{inv3}
If $A\ra_{\rm deg}B$, then
$\dim \Der_{(\al,\be,\ga)}(A) \le \dim \Der_{(\al,\be,\ga)}(B)$ for all $\al,\be,\ga \in \C$.
For $(\al,\be,\ga)=(1,1,1)$ it follows even $\dim \Der (A) < \dim \Der (B)$.
\end{lem}

\begin{defi}
Denote the left and right annihilator of $A$ by
\begin{align*}
\CL (A) & = \{ x\in A\mid x\cdot A=0 \}, \\
\CR (A) & = \{ x\in A\mid A\cdot x=0 \}.
\end{align*}
\end{defi}

We have the following result. 

\begin{lem}\label{inv4}
If $A \ra_{\rm deg} B$ then 
\begin{align*}
\dim A\cdot A & \ge \dim B\cdot B \\ 
\dim \CL (A) & \le \dim \CL (B) \\
\dim \CR (A) & \le \dim \CR (B) 
\end{align*}
\end{lem}

%\newpage

\section{Classification of 3-dimensional Novikov algebras}

A classification of complex $3$-dimensional Novikov algebras can be found in \cite{BAI}. 
However it is not best suited for the purpose of classifying its degenerations.
For this reason we have done an independent classification which is more adapted to
degenerations. More precisely, we first fix a complex $3$-dimensional Lie algebra  
and then classify all Novikov algebras having this Lie algebra as associated Lie algebra. 
This is useful since the Novikov algebra degenerations are a refinement of the associated
Lie algebra degenerations. Secondly, we decrease the number of different families
of Novikov algebras by realizing several algebras as special cases of one family of algebras.
For details on the method of our classification see \cite{BU42}. It should be mentioned that
our classification is equivalent with the one given in \cite{BAI}. We can give an explicit
correspondence. Since the associated Lie algebra of a Novikov algebra must be solvable \cite{BU22}, 
we start with Ja\-cob\-sons list of complex $3$-dimensional, solvable Lie algebras:
\vspace*{0.5cm}
\begin{center}
\begin{tabular}{c|c|c}
$\Lg$ & Lie brackets & $\dim \Der (\Lg)$\\
\hline
$\C^3$ &  $-$ & $9$ \\
\hline
$\Ln_3(\C)$ & $[e_1,e_2]=e_3$ & $6$ \\
\hline
$\Lr_2(\C) \oplus  \C$ & $[e_1,e_2]=e_2$ & $4$ \\
\hline
$\Lr_3(\C)$ & $ [e_1,e_2]=e_2, [e_1,e_3]=e_2+e_3 $ & $4$ \\
\hline
$\Lr_{3,\la}(\C)$ & $[e_1,e_2]=e_2, [e_1,e_3]=\la e_3$ & $4, \la\neq 1$ \\
$\la\neq 0$ & & $6, \la=1$ \\
\end{tabular}
\end{center}
\vspace*{0.5cm}
Among the family $\Lr_{3,\la}(\C)$ we still have isomorphisms. In fact,
$\Lr_{3,\la}(\C)\cong \Lr_{3,\ov{\la}}(\C)$ if and only if $\ov{\la}=\la^{-1}$ or $\ov{\la}=\la$.
We could restrict the parameter $\la\in \C^{\times}$ accordingly in order to avoid repetitions.
It will be more convenient however, to write just $\la\in \C^{\times}$, and keep in mind the
above isomorphisms. Here is our list of complex $3$-dimensional Novikov algebras:
\vspace*{0.5cm}
\begin{center}
\begin{tabular}{c|c|c}
$A$ & Products & $\Lg_A$ \\
\hline
$A_1$ & $-$ & $\C^3$ \\
\hline
$A_2$ & $e_3\cdot e_3=e_3.$ & $\C^3$ \\
\hline
$A_3$ & $e_2\cdot e_2=e_2,\; e_3\cdot e_3=e_3.$  &  $\C^3$ \\
\hline
$A_4$ & $e_1\cdot e_1=e_1,\;e_2\cdot e_2=e_2,\;e_3\cdot e_3=e_3.$ &  $\C^3$ \\
\hline
$A_5$ & $e_2\cdot e_2=e_1.$  & $\C^3$ \\
\hline
$A_6$ & $e_2\cdot e_2=e_1,\; e_3\cdot e_3=e_3.$  & $\C^3$ \\
\hline
$A_7$ & $e_1\cdot e_2=e_1,\; e_2\cdot e_1=e_1,\;e_2\cdot e_2=e_2.$ & $\C^3$ \\
\hline
$A_8$ & $e_1\cdot e_2=e_1,\; e_2\cdot e_1=e_1,\;e_2\cdot e_2=e_2,\;e_3\cdot e_3=e_3.$ & $\C^3$  \\
\hline
$A_9$ & $e_2\cdot e_3=e_1,\; e_3\cdot e_2=e_1.$ & $\C^3$ \\
\hline
$A_{10}$ & $e_2\cdot e_3=e_1,\; e_3\cdot e_2=e_1,\;e_3\cdot e_3=e_2.$ & $\C^3$ \\
\hline
$A_{11}$ & $e_1\cdot e_3=e_1,\; e_2\cdot e_3=e_2,\; e_3\cdot e_1=e_1,$ & $\C^3$ \\
        & $e_3\cdot e_2=e_2,\;e_3\cdot e_3=e_3.$ & \\
\hline
$A_{12}$ & $e_1\cdot e_3=e_1,\;e_2\cdot e_2=e_1,\;e_2\cdot e_3=e_2,$ & $\C^3$ \\
        & $e_3\cdot e_1=e_1,\;e_3\cdot e_2=e_2,\; e_3\cdot e_3=e_3.$ &  \\
\hline
$B_1$ & $e_1\cdot e_1=e_1,\; e_1\cdot e_2=e_2+e_3,$ & $\Ln_3(\C)$ \\
      & $e_1\cdot e_3=e_3,\; e_2\cdot e_1=e_2,\;e_3\cdot e_1=e_3.$ & \\
\hline
$B_2$ & $e_1\cdot e_1=e_1,\; e_1\cdot e_2=e_2+e_3,\;e_1\cdot e_3=e_3,$ & $\Ln_3(\C)$ \\
      & $e_2\cdot e_1=e_2,\;e_2\cdot e_2=e_3,\; e_3\cdot e_1=e_3.$ & \\
\hline
$B_3$ & $e_1\cdot e_2=\frac{1}{2}e_3,\; e_2\cdot e_1=-\frac{1}{2} e_3,\;e_2\cdot e_2=e_3.$  & $\Ln_3(\C)$ \\
\hline
$B_4(\al)$ & $e_1\cdot e_2=\al e_3, \; e_2\cdot e_1=(\al-1)e_3,\;e_2\cdot e_2= e_1.$  & $\Ln_3(\C)$ \\
\hline
$B_5(\be)$ & $e_1\cdot e_2=\be e_3,\; e_2\cdot e_1=(\be-1) e_3.$  & $\Ln_3(\C)$ \\
\hline
$C_1$ & $e_1\cdot e_1=-e_1+e_2,\; e_2\cdot e_1=-e_2,\;e_3\cdot e_3=e_3.$  & $\Lr_2(\C)\oplus \C$ \\
\hline
$C_2$ & $e_1\cdot e_1=-e_1+e_2,\; e_1\cdot e_3=-e_3,$  & $\Lr_2(\C)\oplus \C$ \\
      &  $e_2\cdot e_1=-e_2,\; e_3\cdot e_1=-e_3.$ & \\
\hline
$C_3$ & $e_1\cdot e_1=-e_1+e_2,\; e_2\cdot e_1=-e_2.$  & $\Lr_2(\C)\oplus \C$ \\
\hline
$C_4$ & $e_1\cdot e_1=e_3,\; e_1\cdot e_2=e_2.$  & $\Lr_2(\C)\oplus \C$ \\
\hline
$C_5(\al)$ & $e_1\cdot e_1=\al e_1,\; e_1\cdot e_2=(\al+1) e_2,\;e_2\cdot e_1=\al e_2.$ 
& $\Lr_2(\C)\oplus \C$ \\
\hline
$C_6(\be)$ & $e_1\cdot e_1=\be e_1,\; e_1\cdot e_2=(\be +1) e_2,$
& $\Lr_2(\C)\oplus \C$ \\
      &  $e_2\cdot e_1=\be e_2,\; e_3\cdot e_3=e_3.$ & 
\end{tabular}
\end{center}

\begin{center}
\begin{tabular}{c|c|c}
$A$ & Products & $\Lg_A$ \\
\hline
$C_7(\ga)$ & $e_1\cdot e_1=\ga e_1,\; e_1\cdot e_2=(\ga +1) e_2,$ & $\Lr_2(\C)\oplus \C$ \\ 
$\ga \neq 0$  & $e_1\cdot e_3=\ga e_3,\;e_2\cdot e_1=\ga e_2,\;e_3\cdot e_1=\ga e_3.$  & \\
\hline
$D_1$ & $e_1\cdot e_1=-e_1+e_3,\;  e_1\cdot e_3=e_2,$ & $\Lr_3(\C)$ \\
      & $e_2\cdot e_1=-e_2,\; e_3\cdot e_1=-e_3.$ & \\
\hline
$D_2(\al)$ & $e_1\cdot e_1=\al e_1,\; e_1\cdot e_2=(\al+1) e_2,$  & $\Lr_3(\C)$ \\ 
           & $e_1\cdot e_3=e_2+(\al+1)e_3,\;e_2\cdot e_1=\al e_2,$ & \\
           & $e_3\cdot e_1=\al e_3.$ & \\
\hline
$E_{1,\la}(\al)$ & $e_1\cdot e_1=\al e_1,\; e_1\cdot e_2=(\al+1) e_2,$ & $\Lr_{3,\la}(\C)$ \\
$\la\neq 0$  & $e_1\cdot e_3=(\al+\la)e_3,\; e_2\cdot e_1=\al e_2,$  & \\
 & $e_3\cdot e_1=\al e_3.$ & \\
\hline
$E_{2,\la}$ & $e_1\cdot e_1=- e_1+e_2,\; e_1\cdot e_3=(\la -1) e_3,$  & $\Lr_{3,\la}(\C)$ \\ 
$\la\neq 0$ &  $e_2\cdot e_1=- e_2,\; e_3\cdot e_1=-e_3.$ & \\
\hline
$E_3$ & $e_1\cdot e_1=-\frac{1}{2}e_1+e_3,\; e_1\cdot e_2=\frac{1}{2}e_2,\;e_1\cdot e_3=e_2,$ &
$\Lr_{3,\frac{1}{2}}(\C)$ \\
      & $e_2\cdot e_1=-\frac{1}{2}e_2,\;e_3\cdot e_1=e_2-\frac{1}{2}e_3.$ & \\
\hline
$E_4$ & $e_1\cdot e_1=-e_1+e_2,\; e_1\cdot e_3=-\frac{1}{2}e_3,$ & $\Lr_{3,\frac{1}{2}}(\C)$ \\
      & $e_2\cdot e_1=-e_2,\;e_3\cdot e_1=-e_3,\;e_3\cdot e_3=e_2.$ & \\
\hline
$E_5(\be)$ & $e_1\cdot e_1=\be e_1,\; e_1\cdot e_2=(\be+1)e_2,$ & $\Lr_{3,\frac{1}{2}}(\C)$ \\
           & $e_1\cdot e_3=(\be+\frac{1}{2})e_3,\;e_2\cdot e_1=\be e_2,$ & \\
           & $e_3\cdot e_1=\be e_3,\; e_3\cdot e_3=e_2.$ & \\
\hline
$E_6$ & $e_1\cdot e_1=-\frac{1}{2}e_1,\; e_1\cdot e_2=\frac{1}{2}e_2,$ & $\Lr_{3,\frac{1}{2}}(\C)$ \\
      & $e_1\cdot e_3=e_2,\;e_2\cdot e_1=-\frac{1}{2}e_2,\;e_3\cdot e_1=e_2-\frac{1}{2}e_3.$ &  
\end{tabular}
\end{center}
\vspace*{0.5cm}
The isomorphisms in this list are as follows:
\begin{align*}
B_5(\be) & \simeq B_5(\ov{\be}) \text{ if and only if } \ov{\be}=\be \text{ or } \ov{\be}=1-\be,\\
E_{1,\la}(\al) & \simeq E_{1,\ov{\la}}(\ov{\al}) \text{ if and only if } 
(\ov{\la},\ov{\al})=(\la,\al) \text{ or } (\ov{\la},\ov{\al})=(1/\la,\al /\la).
\end{align*}
Sometimes we want to replace a given algebra by an isomorphic one.
In particular we might replace $E_{1,\la}(-\la)$ by $E_{1,\frac{1}{\la}}(-1)$, and
$B_5(1)$ by $B_5(0)$. \\[0.2cm]
The following table shows the $3$-dimensional Novikov algebras ordered by the dimension of its
derivation algebra:
\vspace*{0.4cm}
\begin{center}
\begin{table}[hbt]
\begin{tabular}{c|l}
$\dim \Der (A)$ & $A$ \\
\hline
$0$ & $A_4$ \\
$1$ & $A_3,\; A_8,\;B_2,\;C_1,\;C_6(\be)_{\be\neq -1},\;E_3,\;E_4,\;E_5(\be)_{\be\neq -1}$ \\
$2$ & $A_6,\;A_7,\;A_{12},\;B_1,\;C_2,\;C_3,\;C_4,\;C_5(\al)_{\al\neq 0,-1},\;C_6(-1),\;
C_7(\ga)_{\ga\neq -1},$  \\
    & $D_1,\;D_2(\al)_{\al\neq -1},\;E_{1,\la\neq 1}(\al)_{\al\neq -1,-\la},\;E_{2,\la\neq 1},\;
E_5(-1),\;E_6$ \\
$3$ & $A_{10},\;B_4(\al),\;C_5(0),\;C_5(-1),\;C_7(-1),\;D_2(-1),\;E_{1,\la\neq 1}(-1)$  \\
$4$ & $A_2,\;A_9,\;A_{11},\;B_3,\; B_5(\be)_{\be\neq \frac{1}{2}},\;E_{1,1}(\al)_{\al\neq -1},\; 
E_{2,1}$  \\
$5$ & $A_5$ \\
$6$ & $B_5(\frac{1}{2}),\;E_{1,1}(-1)$ \\
$9$ & $A_1$ 
\end{tabular}
\caption{Dimension of derivation algebra}
\end{table}
\end{center}

\section{Classification of orbit closures in dimension 3}

A degeneration $A\ra_{\rm deg}B$ of Novikov algebras induces a degeneration 
$\Lg_A\ra_{\rm deg}\Lg_B$ of their associated Lie algebras. 
The Hasse diagram of the $3$-dimensional solvable Lie algebra degenerations is given as
follows:
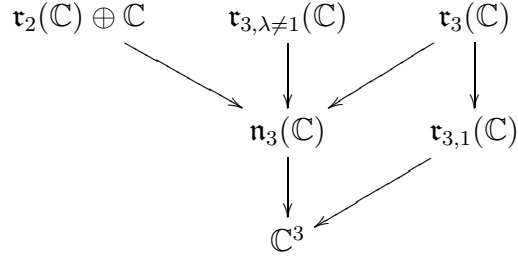
\vspace{0.5cm}
$$
\begin{xy}
\xymatrix{
\Lr_2(\C)\oplus \C\ar[rd]  &  \Lr_{3,\la\neq 1}(\C)\ar[d]  & \Lr_3(\C)\ar[ld]\ar[d] \\
  & \Ln_3(\C)\ar[d] &  \Lr_{3,1}(\C)\ar[ld] \\
 & \C^3 & 
}
\end{xy}
$$
We denote the {\it type} of a Novikov algebra by one of the letters $A,B,C,D,E_{\la}$,
according to its associated Lie algebra $\C^3,\Ln_3(\C), \Lr_2(\C)\oplus \C, \Lr_3(\C), 
\Lr_{3,\la\neq 0}(\C)$. All Novikov algebra de\-ge\-ne\-rations in dimension $3$ arise from the
above Hasse diagram as a refinement. As an example, Novikov algebra degenerations among algebras
of type $A$, symbolically denoted by $A\ra_{\rm deg}A$, refine the improper Lie algebra degeneration 
$\C^3\ra_{\rm deg}\C^3$. Novikov algebra degenerations of type $B\ra_{\rm deg}A$ refine the
degeneration $\Ln_3(\C)\ra_{\rm deg}\C^3$, and so on. 
%\vspace{0.5cm}
%$$
%\begin{xy}
%\xymatrix{
%C\ar[rd]  &  E_{\la\neq 1}\ar[d]  & D\ar[ld]\ar[d] \\
% & B\ar[d] &  E_{\la=1}\ar[ld] \\
% & A & 
%}
%\end{xy}
%$$
Altogether we obtain the following list of possible $13$ types of Novikov 
algebra degenerations:
\vspace*{0.5cm}
\begin{center}
\begin{tabular}{c|l|c|l}
\phantom{ABC} & type & \phantom{ABC} & type \\
\hline
$1$ & $A\ra_{\rm deg}A$ & $8$ & $C\ra_{\rm deg}A$ \\
$2$ & $B\ra_{\rm deg}B$ & $9$ & $E_{\la\neq 1}\ra_{\rm deg}B$ \\
$3$ & $C\ra_{\rm deg}C$ & $10$ & $E_{\la}\ra_{\rm deg}A$ \\
$4$ & $D\ra_{\rm deg}D$ & $11$ & $D\ra_{\rm deg}B$ \\
$5$ & $E_{\la}\ra_{\rm deg}E_{\la}$ & $12$ & $D\ra_{\rm deg}E_{\la = 1}$ \\
$6$ & $C\ra_{\rm deg}B$ & $13$ & $D\ra_{\rm deg}A$ \\
$7$ & $B\ra_{\rm deg}A$ &  & 
\end{tabular}
\end{center}
\vspace*{0.5cm}
For all $13$ types the procedure is as follows:\\[0.2cm]
{\it I : Determination of impossible degenerations of a given type by using
suitable invariants.} \\[0.1cm]
{\it II : Explicit construction of all essential degenerations which are not impossible
by I}.\\[0.2cm]
The essential degenerations of all $13$ types are listed in section $6$. The matrices for 
a degeneration $A\ra_{\rm deg}B$ are chosen in such a way, that they directly 
yield the algebra $B$ from the algebra $A$ in the {\it given bases}, and not only 
a degeneration with isomorphic algebras. \\
The Hasse diagrams for all $13$ cases are listed in section $5$. The normal arrows are
reserved for the new essential degenerations in the diagram, whereas dotted arrows
are used for degenerations which have been given before.
For them we omit sometimes the parameter restrictions, if the diagram becomes too complicated.

\newpage

\begin{prop}\label{4.1}
The orbit closures for type $A\ra_{\rm deg}A$ are given as follows:
\vspace*{0.5cm}
\begin{center}
\begin{tabular}{c|c}
$A$ & $\partial (O(A))$ \\
\hline
$A_1$ & $-$    \\
$A_2$ & $A_1,\,A_5$   \\
$A_3$ & $A_1,\,A_2,\,A_5,\,A_6,\,A_7,\,A_9,\,A_{10}$    \\
$A_4$ & $A_1,\,A_2,\,A_3,\,A_5,\,A_6,\,A_7,\,A_8,\,A_9,\,A_{10},\,A_{11},\,A_{12}$    \\
$A_5$ & $A_1$  \\
$A_6$ & $A_1,\,A_2,\,A_5,\,A_9,\,A_{10}$  \\
$A_7$ & $A_1,\,A_5,\,A_9,\,A_{10}$  \\
$A_8$ & $A_1,\,A_2,\,A_5,\,A_6,\,A_7,\,A_9,\,A_{10},\,A_{11},\,A_{12}$ \\
$A_9$ & $A_1,\,A_5$ \\
$A_{10}$ & $A_1,\,A_5,\,A_9$ \\
$A_{11}$ & $A_1,\,A_5$ \\
$A_{12}$ & $A_1,\,A_5,\,A_9,\,A_{10},\,A_{11}$ 
\end{tabular}
\end{center}
\end{prop}
\vspace*{0.5cm}
\begin{proof}
The algebras $A_1$ to $A_{12}$ are ordered according to their
derivation algebra dimension in table $1$. Degeneration arrows can only go downwards.
The Hasse diagram for this type in section $5$ shows the possible degenerations.
For the essential degenerations the degeneration matrices are listed in section $6$,
from number $6.1$ to number $6.17$. For example, the degeneration $A_4\ra_{\rm deg} A_3$
is obtained as follows. With $g_t^{-1}=\diag (t,1,1)$ and $g_t=\diag (t^{-1},1,1)$
we define $(x\cdot y)_t=g_t(g_t^{-1}(x)\cdot g_t^{-1}(y))$ for $x,y\in A_4$. 
This yields Novikov algebra structures which lie in the orbit of $A_4$.
The algebra $A_4$ has the basis $(e_1,e_2,e_3)$ with products $e_1\cdot e_1=e_1$, $e_2\cdot e_2=e_2$
and $e_3\cdot e_3=e_3$. It follows that $(e_1\cdot e_1)_t=te_1$, $(e_2\cdot e_2)_t=e_2$
and $(e_3\cdot e_3)_t=e_3$. Taking the limit $t\to 0$ we obtain exactly the product for
the algebra $A_3$. \\
It remains to show that all other degenerations, which do not follow from the degenerations
$1$ to $17$, are not possible. First, $A_3$ cannot degenerate to $A_{12}$ or $A_{11}$, because
of $\dim (A_3\cdot A_3)=2$ and $\dim (A_{12}\cdot A_{12})=\dim (A_{11}\cdot A_{11})=3$ and
lemma $\ref{inv4}$. For the same reason $A_6$, $A_7$ and $A_{10}$ cannot degenerate to $A_{11}$.
Furthermore $A_7$ cannot degenerate to $A_2$ because $c_{1,1}(A_7)=2$, but $c_{1,1}(A_2)=1$,
see lemma $\ref{inv1}$. Also $A_{12}$ cannot degenerate to $A_2$ because $c_{1,1}(A_{12})=3$.
Finally $A_{10}$ cannot degenerate to $A_2$ because $A_{10}$ is complete and $A_2$ is not.
Here we may use lemma $\ref{inv2}$ with the operator polynomial $T(x)=R(x)^3$.
\end{proof}

\begin{prop}\label{4.2}
The orbit closures for type $B\ra_{\rm deg}B$ are given as follows:
\vspace*{0.5cm}
\begin{center}
\begin{tabular}{c|c}
$A$ & $\partial (O(A))$ \\
\hline
$B_1$ & $B_4(0),\,B_5(0)$    \\
$B_2$ & $B_1,\,B_3,\,B_4(\al),\,B_5(\be)$   \\
$B_3$ & $B_5(\frac{1}{2})$    \\
$B_4(\al)_{\al\neq \frac{1}{2}}$ & $B_5(\al)_{\al\neq \frac{1}{2}}$  \\
$B_4(\frac{1}{2})$ & $B_3,\,B_5(\frac{1}{2})$  \\
$B_5(\be)$ & $-$
\end{tabular}
\end{center}
\end{prop}
\vspace*{0.5cm}
\begin{proof}
By table $1$, $B_1$ can only degenerate to $B_3$, $B_4(\al)$ or $B_5(\be)$. 
Suppose that $B_1$ degenerates to an algebra $B_4(\al)$. Since $B_1$ satisfies the operator 
identity $T(x)=(R(x)-\frac{1}{3}\tr (R(x))\id)^2=0$, this must hold also for $B_4(\al)$.
But this is true if and only if $\al=0$. And indeed, there is a degeneration 
$B_1\ra_{\rm deg}B_4(0)$, see $6.20$. \\
The algebra $B_1$ cannot degenerate to $B_3$. To see this
we use lemma $\ref{2.7}$, and lemma $3.4$ from \cite{BU36}. We may assume that the matrix
$g_t^{-1}=(f_{ij}(t))$ has upper-triangular form. For convenience we consider a new basis
for $B_1$ by interchanging $e_1$ and $e_3$ from the old basis. Then the operators $L(x)$ for $B_1$ 
are simultaneously upper-triangular as well, and it is easy to compute the orbit of $B_1$ via
$g_t$ and $g_t^{-1}$. In fact we obtain $(e_i\cdot e_j)_t=0$ for $1\le i,j\le 2$ and
$(e_1\cdot e_3)_t=f_{33}(t)e_1$, $(e_2\cdot e_3)_t=f_{33}(t)e_2$. If not $\lim_{t\to 0}f_{33}(t)=0$
then obviously $\dim B\cdot B\ge 2$ for all algebras $B$ in the obit closure of $B_1$. In this case
$B_3$ cannot lie in the orbit closure because of $\dim (B_3\cdot B_3)=1$. However, in the other
case with $\lim_{t\to 0}f_{33}(t)=0$ all algebras $B$ in the orbit closure satisfy
$\dim \CL (B)\ge 2$, because then $e_1$ and $e_2$ are in the left annihilator. Also in this case 
$B_3$ cannot lie in the orbit closure since $\dim \CL (B_3)= 1$. \\
Suppose that $B_1$ degenerates to an algebra $B_5(\be)$. We repeat the above argument. Since
$\dim B_5(\be)\cdot B_5(\be)=1$ for all $\be$ we may assume that $\lim_{t\to 0}f_{33}(t)=0$.
We have $\dim \CL (B_5(\be))\ge 2$ if and only if $\be=0$ or $\be=1$. Hence $B_1$ can only degenerate
to $B_5(0)\simeq B_5(1)$, which is indeed the case, see $6.20$ and $6.21$. \\
The algebra $B_2$ degenerates to every other algebra of class $B$. We have the degenerations
$B_2\ra_{\rm deg}B_1\ra_{\rm deg} B_4(0)\ra_{\rm deg}B_5(0)$, see $6.18$, $6.20$, 
$B_2\ra_{\rm deg}B_4(\al)\ra_{\rm deg}B_5(\al)$ for $\al\neq 0$, see $6.19$ and 
$B_2\ra_{\rm deg}B_4(\frac{1}{2})\ra_{\rm deg}B_3$, see $6.22$. \\
By table $1$, $B_3$ can only degenerate to  $B_5(\frac{1}{2})$. This is possible, see $6.23$. \\
Suppose that $B_4(\al)$ degenerates to $B_3$. Since $\dim \Der_{(0,1,0)}(B_3)=3$, but
 $\dim \Der_{(0,1,0)}(B_4(0))=6$, $B_4(\al)$ cannot degenerate to $B_3$ for $\al=0$ by lemma
$\ref{inv3}$. It is also not possible for $\al=1$, because $\dim \Der_{(0,0,1)}(B_3)=3$ and
$\dim \Der_{(0,0,1)}(B_4(1))=6$. Assume now  $\al\neq 0,1$. Then $\dim \Der_{(0,1,\frac{\al}{1-\al})}(B_3) 
= 3$ for $\al\neq \frac{1}{2}$ and $\dim \Der_{(0,1,\frac{\al}{1-\al})}(B_4(\al)) = 4$. 
This only leaves $\al=\frac{1}{2}$, and there is a degeneration $B_4(\frac{1}{2})
\ra_{\rm deg} B_3$, see $6.22$.
The same argument shows that $B_4(\al)$ degenerates to $B_5(\be)$ if and only if $\al=\be$
or $\al=1-\be$. Note that for $\al,\be \neq 0,1$ we have $\dim \Der_{(0,1,\frac{\al}{1-\al})}(B_5(\be))=3$,
provided $\al\neq \be,1-\be,\frac{1}{2}$. There is a degeneration 
$B_4(\al)\ra_{\rm deg} B_5(\al)\simeq B_5(1-\al)$, see $6.21$. \\
Again using lemma $\ref{inv3}$ we see that $B_5(\be)$ degenerates to $B_5(\frac{1}{2})$ only for 
$\be=\frac{1}{2}$: for $\ov{\be}\neq 1$ we have $\dim \Der_{(1,\ov{\be},1)}(B_5(\frac{1}{2}))=3$, and
$\dim \Der_{(0,\frac{1-\be}{\be},1)}(B_5(\be))=4$ for $\be\neq 0,\frac{1}{2},1$. Moreover   
we have $\dim \Der_{(1,0,1)}(B_5(0))=5$. Hence $B_5(1)\simeq B_5(0)$ cannot degenerate to 
$B_5(\frac{1}{2})$. 
\end{proof}

\begin{prop}
The orbit closures for type $C\ra_{\rm deg}C$ are given as follows:
\vspace*{0.5cm}
\begin{center}
\begin{tabular}{c|c}
$A$ & $\partial (O(A))$ \\
\hline
$C_1$ & $C_2,\,C_3,\,C_5(-1),\,C_6(-1),\,C_7(-1)$    \\
$C_2$ & $C_7(-1)$   \\
$C_3$ & $C_5(-1)$    \\
$C_4$ & $C_5(0)$    \\
$C_5(\al)$ & $-$  \\
$C_6(\be)_{\be\neq 0}$ & $C_5(\be)_{\be\neq 0},\,C_7(\be)_{\be\neq 0}$  \\
$C_6(0)$ & $C_4,\,C_5(0)$    \\
$C_7(\ga)_{\ga\neq 0}$ & $-$
\end{tabular}
\end{center}
\end{prop}
\vspace*{0.5cm}
\begin{proof}
The algebra $C_1$ satisfies the operator identity $T(x)=L(x)^2-L(x)R(x)=0$.
This identity holds in $C_5(\al)$, $C_6(\be)$ or $C_7(\ga)$ exactly for
$\al=-1$, $\be=-1$ or $\ga=-1$. It does not hold for $C_4$. Hence $C_1$ can only
degenerate to $C_2$,$C_3$,$C_5(-1)$,$C_6(-1)$ and $C_7(-1)$. Indeed, all of these
degenerations exist, see $6.24-6.32$ and the Hasse diagram in section $5$.\\
The algebras $C_2$ and $C_3$ also satisfy the above operator identity, so that they
only could degenerate to $C_5(-1)$ or $C_7(-1)$, see table $1$. Since $d_{1,1}(C_2)=3$, but
$d_{1,1}(C_5(-1))=2$, $C_2$ cannot degenerate to $C_5(-1)$ by lemma $\ref{inv1}$. 
On the other hand, the degeneration $C_2\ra_{\rm deg}C_7(-1)$ is possible, see $6.30$.
Since $d_{1,1}(C_3)=2$, but $d_{1,1}(C_7(-1))=3$, only $C_3\ra_{\rm deg}C_5(-1)$ remains.
This is possible, see $6.31$. \\
The algebra $C_4$ can only degenerate to $C_7(-1)$, $C_5(-1)$ or $C_5(0)$, see table $1$.
Since $C_4$ is complete, only the complete algebra $C_5(0)$ remains. Indeed, there is
a degeneration $C_4\ra_{\rm deg}C_5(0)$, see $6.32$. For $\be\neq 0$ consider the operator polynomial
\[
T_{\be}(x)=L(x)^2-L(x)R(x)-\frac{\be+1}{\be}R(x)L(x)+\frac{\be+1}{\be}R(x)^2.
\]
We have $T_{\al}(x)=0$ for $C_5(\al)$ with $\al\neq 0$. This holds for $C_7(\ga)$ if and only
if $\ga=\al$. The algebra $C_5(\al)$ with $\al\neq 0,-1$ can only degenerate to 
$C_7(-1)$, $C_5(-1)$ or $C_5(0)$, see table $1$. None of these is possible because of the above
operator identity. By table $1$, there are no proper degenerations for $C_5(0)$ and $C_5(-1)$
within type $C$. \\
The algebra $C_6(\be)$ degenerates to $C_5(\al)$ and to $C_7(\ga)$ for $\al=\be$ and $\ga=\be$, 
see $6.28$, $6.29$. We want to show that there are no degenerations possible if not $\al=\be=\ga$.
For $\be \neq 0$ this follows from the fact, that $C_6(\be)$ satisfies $T_{\be}(x)=0$,
$C_5(\al)$ satisfies this identity if and only if $\al=\be$, and $C_7(\ga)$ satisfies  it if and only if 
$\ga=\be$. For $\be=0$ we consider the operator polynomial $T_0(x)=R(x)L(x)-R(x)^2$. In $C_6(0)$
we have $T_0(x)=0$, and the only other algebras of type $C$ satisfying this identity are
$C_4$, $C_5(0)$ and $C_7(0)$. However, for the family $C_7(\ga)$ we have required $\ga\neq 0$.
The degenerations $C_6(0)\ra_{\rm deg}C_4\ra_{\rm deg}C_5(0)$ are possible, see $6.27$ and $6.32$.
\end{proof}

\begin{prop}\label{4.4}
The orbit closures for type $D\ra_{\rm deg}D$ are given as follows:
\vspace*{0.5cm}
\begin{center}
\begin{tabular}{c|c}
$A$ & $\partial (O(A))$ \\
\hline
$D_1$ & $D_2(-1)$    \\
$D_2(\al)$ & $-$   
\end{tabular}
\end{center}
\end{prop}
\vspace*{0.5cm}
\begin{proof}
By table $1$, the algebra $D_1$ can only degenerate to $D_2(-1)$. This is possible, see $6.33$. 
It remains to determine the $\al\in \C$ for which there is a degeneration
$D_2(\al)\ra_{\rm deg}D_2(-1)$. We have $c_{i,j}(D_2(-1))=1$ for all $i,j\in \N$ and

\[
c_{i,j}(D_2(\al))=\frac{(\al^i+2(\al+1)^i)(\al^j+2(\al+1)^j)}{\al^{i+j}+2(\al+1)^{i+j}}
\]
in general. Comparing the invariants $c_{1,1}$ we obtain 
\[
\frac{(3\al+2)^2}{\al^2+2(\al+1)^2}=1,
\] 
provided the denominator is nonzero. In this case we obtain $(\al+1)(3\al+1)=0$. Using $c_{1,2}$ 
we see that $\al^2+2(\al+1)^2=0$ and $\al=-\frac{1}{3}$ are both impossible. Hence we obtain $\al=-1$.
\end{proof}

\begin{prop}\label{4.5}
The orbit closures for type $E_{\la}\ra_{\rm deg}E_{\la}$, $\la\neq 0$ are given as follows:
\vspace*{0.5cm}
\begin{center}
\begin{tabular}{c|c}
$A$ & $\partial (O(A))$ \\
\hline
$E_{1,\la}(\al)$ & $-$    \\
$E_{2,\la}$ & $E_{1,\la}(-1)$   \\
$E_3$ & $E_{1,2}(-1),\,E_{2,2},\,E_6$    \\
$E_4$ & $E_{1,\frac{1}{2}}(-1),\,E_{2,\frac{1}{2}},\,E_5(-1)$    \\
$E_5(\be)_{\be\neq -\frac{1}{2},-1}$ & $E_{1,\frac{1}{2}}(\be)_{\be\neq -\frac{1}{2},-1}$  \\
$E_5( -\frac{1}{2})$ & $E_{1,2}(-1),\,E_6$   \\
$E_5(-1)$ & $E_{1,\frac{1}{2}}(-1)$   \\
$E_6$ & $E_{1,2}(-1)$  
\end{tabular}
\end{center}
\end{prop}
\vspace*{0.5cm}
\begin{proof}
For $\la=1$ the algebra $E_{1,\la}(\al)$ can only degenerate to $E_{1,1}(-1)$ by table $1$.
Since $c_{i,j}(E_{1,1}(\al))=c_{i,j}(D_2(\al))$ it follows $\al=-1$ as above.
The case $\la=-1$ can be reduced to the case $\la=1$ because of $E_{1,-1}(\al)\simeq E_{1,1}(-\al)$.
If $E_{1,\la}(\al)$ degenerates to $E_{1,\ov{\la}}(\ov{\al})$, then we obtain an
associated Lie algebra degeneration $\Lr_{3,\la}(\C)\ra \Lr_{3,\ov{\la}}(\C)$, see lemma 
$\ref{inv4}$, so that $\ov{\la}=\la$ or $\ov{\la}=\frac{1}{\la}$.
For $\al\neq 0$ consider the operator polynomial
\begin{align*}
T_{\al,\la}(x) & = L(x)^3-L(x)^2R(x)-\frac{\al+\la}{\al}L(x)R(x)L(x)
-\frac{\al+1}{\al}R(x)L(x)^2 \\[0.2cm]
 & + \frac{3\al+2\la+1}{\al}R(x)L(x)R(x)+\frac{\al+\la}{\al^2}R(x)^2L(x)
-\frac{(\al+\la)(\al+1)}{\al^2}R(x)^3.
\end{align*}
We have $T_{\al,\la}(x)=0$ for $E_{1,\la}(\al)$. The identity holds for $E_{1,\ov{\la}}(\ov{\al})$ 
with $\ov{\al}\neq 0$ if and only if
\begin{align*}
0 & = (\al-\ov{\al}\la)(\al-\ov{\al}), \\
0 & = (\al \ov{\la}-\ov{\al}\la)(\al \ov{\la}-\ov{\al}).
\end{align*}
As we have seen, either $\ov{\la}=\la$ or $\ov{\la}=\frac{1}{\la}$. Assume first
$\ov{\la}=\la$. Then either $\ov{\al}=\al$ or $\la^2=1$. The cases $\la=\pm 1$ have been
treated above. For $\ov{\la}=\frac{1}{\la}$ we obtain either $\ov{\al}=\al$ or 
$\ov{\al}=\frac{\al}{\la}$. The second case yields an isomorphic algebra since
$E_{1/\la}(\al/\la)\simeq E_{1,\la}(\al)$. Together this shows that there is no proper degeneration 
$E_{1,\la}(\al)\ra_{\rm deg}E_{1,\ov{\la}}(\ov{\al})$ for $\al,\ov{\al}\neq 0$. For $\al=0$
the same argument can be applied to the operator $T_{0,\la}(x)=R(x)L(x)-R(x)^2$, which is always 
zero for $E_{1,\la}(0)$, and zero for $E_{1,\ov{\la}}(\ov{\al})$ if and only if $\ov{\al}=0$. \\
The algebra $E_{2,\la}$ may degenerate to $E_{2,1}$ or $E_{1,\ov{\la}}(\al)$ by table $1$.
In the first case we obtain $\la=1$ by lemma $\ref{inv4}$. For the second case consider the
operator polynomial 
\[
S_{\la}(x)=L(x)^3-L(x)^2R(x)+(\la-1)R(x)L(x)^2-(\la-1)R(x)L(x)R(x).
\]
We have $S_{\la}(x)=0$ for $E_{2,\la}$. This identity holds for $E_{1,\ov{\la}}(\al)$ if and only if
\begin{align*}
0 & = (\al\la+1)(\al+1), \\
0 & = (\al\la+ \ov{\la})(\al+\ov{\la}).
\end{align*}
Again we have two cases, either $\ov{\la}=\la$ or $\ov{\la}=\frac{1}{\la}$.
In the first case it follows $\al=-1$, and there is indeed a degeneration
$E_{2,\la}\ra_{\rm deg}E_{1,\la}(-1)$, see $6.41$. In the second case we obtain 
either $\al=-1$, $\la^2=1$, or $\al=-\frac{1}{\la}$. Since $E_{1,\frac{1}{\la}}(-\frac{1}{\la})\simeq
E_{1,\la}(-1)$ we are again back to the case $\al=-1$. \\
The algebra $E_3$ may degenerate to $E_6$, $E_{2,\la}$, $E_5(-1)$ or $E_{1,\la}(\al)$ by table $1$.
Moreover, $\la=2$ or $\la=\frac{1}{2}$ by lemma $\ref{inv4}$. For $E_3$ we have 
\begin{align*}
S_2(x) & = T_{-\frac{1}{2},\frac{1}{2}}(x)= L(x)^3-L(x)^2R(x)+R(x)L(x)^2-R(x)L(x)R(x)=0.  
\end{align*}
This identity holds for $E_{2,\la}$ if and only if $(\la-1)(\la-2)=0$. It follows $\la=2$,
and indeed there is a degeneration $E_3\ra_{\rm deg} E_{2,2}$, see $6.36$. 
The identity $S_2(x)=0$ holds for $E_{1,\la}(\al)$ if and only if
\begin{align*}
0 & = (2\al+1)(\al+1), \\
0 & = (2\al+\la)(\al+\la).
\end{align*}
For $\la=2$ we obtain $\al=-1$, and indeed there is a degeneration
$E_3\ra_{\rm deg} E_{2,2}\ra_{\rm deg} E_{1,2}(-1)$. For $\la=\frac{1}{2}$ we obtain
$\al=-\frac{1}{2}$, but as before we have $E_{1,\frac{1}{2}}(-\frac{1}{2})\simeq E_{1,2}(-1)$.
Finally, the identity $S_2(x)=0$ holds for $E_6$, but not for $E_5(-1)$. There is a degeneration
$E_3 \ra_{\rm deg}E_6$, see $6.35$. \\
The algebra $E_4$ may degenerate to $E_6$, $E_{2,\la}$, $E_5(-1)$ or $E_{1,\la}(\al)$ by table $1$.
It satisfies the identity $T_{\al,\la}(x)=0$ with $(\al,\la)=(-1,\frac{1}{2})$. This identity holds 
for $E_5(-1)$ but not for $E_6$. It holds for $E_{2,\la}$ if and only if $(2\la-1)(\la-1)=0$.
Again $\la=1$ is impossible, but the degenerations $E_4 \ra_{\rm deg} E_{2,\frac{1}{2}}$ and
$E_4 \ra_{\rm deg} E_5(-1)$ exist, see $6.37$, $6.38$. The identity holds for $E_{1,\la}(\al)$
if and only if 
\begin{align*}
0 & = (\al+2)(\al+1), \\
0 & = (\al+2\la)(\al+\la).
\end{align*}
A discussion as above shows that the only possibility is a degeneration
$E_4 \ra_{\rm deg}E_{1,\frac{1}{2}}(-1)$, which exists, see $6.38$ and $6.39$.\\
An algebra $E_5(\be)$ may degenerate to $E_6$, $E_{2,\la}$, $E_5(-1)$ or $E_{1,\la}(\al)$ 
by table $1$, with $\la=2$ or $\la=\frac{1}{2}$.
For $\be=0$ the algebra $E_5(0)$ is complete and therefore can only degenerate to 
$E_{1,\frac{1}{2}}(0)$, which is indeed the case, see $6.39$. For $\be\neq 0$ we have
$T_{\al,\la}(x)=0$ with $(\al,\la)=(\be,\frac{1}{2})$. This identity holds for 
$E_6$ if and only if $\be=-\frac{1}{2}$, and there exists a degeneration 
$E_5(-\frac{1}{2})\ra_{\rm deg} E_6$, see $6.34$. The identity is satisfied
for $E_{2,\la}$ if and only if 
\begin{align*}
0 & = (2\be+1)(\be+1), \\
0 & = (2\la\be+1)(\be\la+1).
\end{align*}
For $\la=\frac{1}{2}$ we obtain $\be=-1$ or $\be=-\frac{1}{2}$. However $\be=-1$ is not possible 
because of table $1$, and $\be=-\frac{1}{2}$ is also not possible since 
$E_5(-\frac{1}{2})$ satisfies the identity 
\[
T(x)=L(x)^3-L(x)R(x)L(x)+R(x)L(x)^2-R(x)^2L(x)=0,
\]
whereas it does not hold for any algebra $E_{2,\la}$. 
For $\la=2$ we obtain $\be=-\frac{1}{2}$. Again $E_5(-\frac{1}{2})$ cannot degenerate to
$E_{2,2}$, since $E_{2,2}$ does not satisfy the identity $T(x)=0$. The identity $T_{\be,\frac{1}{2}}(x)=0$
holds for $E_{1,\la}(\al)$ if and only if
\begin{align*}
0 & = (\al-\be)(\al-2\be), \\
0 & = (\al-\la\be)(\al-2\la\be).
\end{align*}
If $\la=\frac{1}{2}$ then $\al=\be$. If $\la=2$ then $\al=\be$ or $\al=2\be$. Therefore
the only possible degeneration of $E_5(\be)$ to $E_{1,\la}(\al)$
is $E_5(\be)\ra_{\rm deg}E_{1,\frac{1}{2}}(\be)\simeq E_{1,2}(2\be)$, which 
indeed exists, see $6.39$. Using again $T_{\be,\frac{1}{2}}(x)=0$ we see that
$E_5(\be)$ cannot degenerate to $E_5(-1)$ except for $\be=-1$. \\
Finally $E_6$ may degenerate to an algebra $E_{1,\la}(-1)$ with $\la=2$ or $\la=\frac{1}{2}$.
As we have seen above, $E_6$ satifies the operator identity $S_2(x)=0$. This holds
for $E_{1,\la}(-1)$ if and only if $(\la-2)(\la-1)=0$. Hence we only can have a degeneration
$E_6\ra_{\rm deg} E_{1,2}(-1)$, which indeed exists, see $6.40$.
\end{proof}

\begin{prop}\label{4.6}
The orbit closures for type $C\ra_{\rm deg}B$ are given as follows:
\vspace*{0.5cm}
\begin{center}
\begin{tabular}{c|c}
$A$ & $\partial (O(A))$ \\
\hline
$C_1$ & $B_3,\,B_4(\al),\,B_5(\al)$    \\
$C_2$ & $B_4(0),\,B_5(0)$   \\
$C_3$ & $B_4(1),\,B_5(0)$    \\
$C_4$ & $B_4(0),\,B_5(0)$    \\
$C_5(\al)_{\al\neq 0,-\frac{1}{2},-1}$ & $B_4(-\al),\,B_5(-\al),\; \al\neq 0,-\frac{1}{2},-1$.\\
$C_5(0)$ & $B_5(0)$  \\
$C_5(-\frac{1}{2})$ & $B_3,\,B_4(\frac{1}{2}),\,B_5(\frac{1}{2})$  \\
$C_5(-1)$ & $B_5(0)$  \\
$C_6(\be)_{\be\neq -1}$ & $B_3,\,B_4(\be),\,B_5(\be)$  \\
$C_6(-1)$ & $B_5(0)$    \\
$C_7(\ga)_{\ga\neq 0,-1}$ & $B_4(0),\,B_5(0)$ \\
$C_7(-1)$ & $B_5(0)$ 
\end{tabular}
\end{center}
\end{prop}
\vspace*{0.5cm}
\begin{proof}
By table $1$, $C_1$ may degenerate to $B_1$, $B_3$, $B_4(\al)$ or $B_5(\be)$.
Since $L(x)^2-L(x)R(x)=0$ in $C_1$, but not in $B_1$, there is no degeneration
from $C_1$ to $B_1$. There is a degeneration $C_1\ra_{\rm deg}B_4(\al)$ for $\al\neq 0,1$,
see $6.45$. This is also true for $\al=0$ and $\al=1$ by transitivity. We have
$C_1\ra_{\rm deg}C_3\ra_{\rm deg}B_4(1)$, see $6.49$, and $C_1\ra_{\rm deg}C_2\ra_{\rm deg}B_4(0)$,
see $6.50$. Moreover we have $C_1\ra_{\rm deg}B_4(\frac{1}{2})\ra_{\rm deg}B_3$ and
$C_1\ra_{\rm deg}B_4(\al)\ra_{\rm deg}B_5(\al)$. \\
The algebra $C_2$ may degenerate to $B_3$, $B_4(\al)$ or $B_5(\be)$ by table $1$.
It cannot degenerate to $B_3$ by exactly the same argument why $B_1$ cannot degenerate to
$B_3$, see proposition $\ref{4.2}$: all $B\in \ov{O(C_2)}$ either satisfy $\dim (B\cdot B)\ge 2$
or $\dim \CL(B)\ge 2$. However we have $\dim B_3\cdot B_3=\dim \CL(B_3)=1$.
Similarly $C_2$ degenerates to $B_5(\be)$ if and only if $\be=0$ or $\be=1$. Since
$(R(x)-\frac{1}{3}\tr (R(x))\id)^2=0$ for $C_2$, and this holds for $B_4(\al)$ if and only
if $\al=0$, it follows that $C_2$ only degenerates to $B_4(0)$ in this family.
Similarly, $C_2$ cannot degenerate to $B_5(\be)$ for $\be\neq 0,1$. We have 
$C_2\ra_{\rm deg} B_4(0)\ra_{\rm deg} B_5(0)$, see $6.50$. \\
The algebra $C_3$ satisfies $\dim \CR(C_3)=2$. It can only degenerate to algebras
$B$ with $\dim \CR(B)\ge 2$, see lemma $\ref{inv4}$. Using table $1$ we see that
the only candidates are $B_4(1)$ and $B_5(0)$. We have $C_3\ra_{\rm deg} B_4(1)\ra_{\rm deg} B_5(0)$, 
see $6.49$. \\
The algebra $C_4$ satisfies $\dim \CL(C_4)=2$. Among the algebras $B_3$, $B_4(\al)$, $B_5(\be)$
only $B_4(0)$ and $B_5(0)$ satisfy  $\dim \CL(B)\ge 2$. In this case we have the
degenerations $C_4\ra_{\rm deg} B_4(0)\ra_{\rm deg} B_5(0)$, see $6.46$. \\
Suppose that $C_5(\al)$ degenerates to $B_4(\ov{\al})$. We may assume that $\al\neq 0,-1$ by 
table $1$. Then $\dim \Der_{(0,1,\frac{-\al}{1+\al})}(C_5(\al))=4$ and
$\dim \Der_{(0,1,\frac{-\al}{1+\al})}(B_4(\ov{\al}))=3$ for $\ov{\al}\neq 0,1$ and $\al+\ov{\al}\neq 0$.
This only leaves $\ov{\al}=-\al$ for $\al\neq 0,-1$, and indeed such a degeneration
exists, see $6.47$. \\
Suppose that $C_5(\al)$ degenerates to $B_3$. Since $\dim \Der_{(0,1,0)}(C_5(0))=6$ and
$\dim \Der_{(0,1,0)}(B_3)=3$, we cannot have $\al=0$. Also $\al=-1$ is impossible since 
$\dim \Der_{(0,0,1)}(C_5(-1))=6$ and $\dim \Der_{(0,0,1)}(B_3)=3$. For the case $\al\neq 0,-1$
we have $\dim \Der_{(0,1,\frac{-\al}{1+\al})}(C_5(\al))=4$ and
$\dim \Der_{(0,1,\frac{-\al}{1+\al})}(B_3)=3$ if $\al\neq -\frac{1}{2}$. This only leaves
$\al= -\frac{1}{2}$, and indeed there is a degeneration $C_5(-\frac{1}{2})\ra_{\rm deg}
B_4(\frac{1}{2})\ra_{\rm deg} B_3$ by transitivity. \\
An algebra $C_5(\al)$ degenerates to an algebra $B_5(\be)$ only if $\be=-\al$ or $\be=1+\al$,
and a degeneration $C_5(\al)\ra_{\rm deg} B_4(-\al)\ra_{\rm deg} B_5(-\al)\simeq B_5(1+\al)$ is 
possible, see $6.47$, $6.51$ and $6.52$. For $\al=0$ and $\al=-1$ this follows again by using 
$(0,1,0)$-derivations and $(0,0,1)$-derivations as above. If $\al\neq 0,-\frac{1}{2},-1$ we have
$\dim \Der_{(0,1,\frac{-\al}{1+\al})}(B_5(\be))=3$ for $\be\neq 0,1$ and $(\al+\be)(\al+1-\be)\neq 0$.
Since this space has dimension $4$ for $C_5(\al)$, this shows the claim except for 
$\al=-\frac{1}{2}$. We need to show that $C_5(-\frac{1}{2})$ degenerates to $B_5(\be)$ if and only
if $\be=\frac{1}{2}$. A degeneration  $C_5(-\frac{1}{2})\ra_{\rm deg}B_5(\be)$ implies a degeneration
$J_{C_5(-\frac{1}{2})}\ra_{\rm deg}J_{B_5(\be)}$ by lemma  $\ref{inv5}$.
We have $J_{C_5(-\frac{1}{2})}\simeq A_2$ and $J_{B_5(\be)}\simeq A_9$ for
$\be\neq \frac{1}{2}$. Since $A_2$ cannot degenerate to $A_9$, see table $1$, the claim follows. \\
By table $1$, $C_6(\be)$ with $\be \neq -1$ may degenerate to $B_1$, $B_4(\al)$, $B_3$ or
$B_5(\ov{\be})$.
For $\be\neq 0$ consider the operator polynomial
\[
T_{\be}(x)=L(x)^2-L(x)R(x)-\frac{\be+1}{\be}R(x)L(x)+\frac{\be+1}{\be}R(x)^2.
\]
It is zero for $C_6(\be)$ but not for $B_1$. This shows that $C_6(\be)$ cannot degenerate to
$B_1$ for $\be\neq 0$. For $C_6(0)$ the operator $T(x)=R(x)L(x)-R(x)^2$ is zero, but not for
$B_1$. Hence $C_6(0)$ cannot degenerate to $B_1$ either. A degeneration  $C_6(\be)\ra_{\rm deg} 
B_4(\al)\ra_{\rm deg}B_5(\al)$ for $\be\neq -1$ is possible, see $6.42-6.44$.
The algebra $C_6(-1)$ can only degenerate to $B_5(0)\simeq B_5(1)$ since $\dim \Der_{(1,1,0)}
(C_6(-1))=5$, but the space of $(1,1,0)$-derivations for $B_3$, $B_4(\al)$ and
$B_5(\ov{\be})$ with $\ov{\be}\neq 0,1$ is only $3$-dimensional. There is a degeneration
$C_6(-1)\ra_{\rm deg} C_5(-1)\ra_{\rm deg}B_5(0)$ by transitivity and $6.52$. \\
By table $1$, $C_7(\ga)$ with $\ga \neq 0,-1$ may degenerate to $B_3$, $B_4(\al)$ or
$B_5(\be)$. The same argument whcih we have used for $C_2$ shows that only $
B_4(0)$ and $B_5(0)$ cannot be excluded. Indeed, there is a degeneration 
$C_7(\ga)\ra_{\rm deg} B_4(0)\ra_{\rm deg}B_5(0)$ for $\ga\neq 0,-1$, see $6.48$.
By definition $\ga\neq 0$, but the case $\ga=-1$ still has to be considered. We have
$\dim \Der_{(1,1,0)}(C_7(-1))=5$, but $\dim \Der_{(1,1,0)}(B)=3$ for all other algebras in question,
except $B_5(0)\simeq B_5(1)$. There is a degeneration $C_7(-1)\ra_{\rm deg} B_5(0)$, see $6.53$.
\end{proof}

\begin{prop}\label{4.7}
The orbit closures for type $B\ra_{\rm deg}A$ are given as follows:
\vspace*{0.5cm}
\begin{center}
\begin{tabular}{c|c}
$A$ & $\partial (O(A))$ \\
\hline
$B_1$ & $A_1,\,A_5,\,A_{11}$    \\
$B_2$ & $A_1,\,A_5,\,A_9,\,A_{10},\,A_{11},\,A_{12}$   \\
$B_3$ & $A_1,\,A_5$    \\
$B_4(\al)$ & $A_1,\,A_5$  \\
$B_5(\be)_{\be\neq \frac{1}{2}}$ & $A_1,\,A_5$  \\
$B_5(\frac{1}{2})$ & $A_1$  
\end{tabular}
\end{center}
\end{prop}
\vspace*{0.5cm}
\begin{proof}
By table $1$, $B_1$ can only degenerate to $A_1$, $A_2$, $A_5$, $A_9$, $A_{10}$ or  $A_{11}$.
Since $c_{1,1}(B_1)=3$ and $c_{1,1}(A_2)=1$, $B_1$ cannot degenerate to $A_2$. Also $B_1$
cannot degenerate to $A_9$ by the same argument which showed that $B_1$ cannot degenerate 
to $B_3$. Then $B_1$ also cannot degenerate to $A_{10}$ since $A_{10}\ra_{\rm deg} A_9$. 
On the other hand we have the degenerations $B_1\ra_{\rm deg} A_{11}\ra_{\rm deg} A_5\ra_{\rm deg} A_1$, 
see $6.55$. \\
The algebra $B_2$ cannot degenerate to $A_6$, $A_7$ and $A_2$ since $c_{1,1}(B_2)=3$,
$c_{1,1}(A_2)=c_{1,1}(A_6)=1$, $c_{1,1}(A_7)=2$. By table $1$, only $A_1,\,A_5,\,A_9,\,A_{10},
\,A_{11},\,A_{12}$ remain. These degenerations exist, by $B_2\ra_{\rm deg}A_{12}$ and
transitivity, see $6.54$. \\
By table $1$, $B_3$ can only degenerate to $A_1$ or $A_5$. This is possible, see $6.56$. \\
An algebra $B_4(\al)$ may degenerate to $A_2$, $A_9$, $A_{11}$, $A_5$ or $A_1$, see table $1$.
Since all $B_4(\al)$ are complete, $A_2$ and $A_{11}$ can be excluded. There is no degeneration
to $A_9$, since for $\al \neq 0,1$, $\dim \Der_{(0,1,\frac{\al}{1-\al})}(B_4(\al))=4$ and
$\dim \Der_{(0,1,\frac{\al}{1-\al})}(A_9)=3$. For $\al=0$ consider $(0,1,0)$-derivations
and for $\al=1$ consider $(0,0,1)$-derivations. The degenerations $B_4(\al)\ra_{\rm deg} A_5\ra_{\rm deg}A_1$
are possible: for $\al\neq \frac{1}{2}$ we have $B_4(\al)\ra_{\rm deg}B_5(\al)\ra_{\rm deg} A_5$, 
see $6.57$, and for $\al=\frac{1}{2}$ we have $B_4(\frac{1}{2})\ra_{\rm deg}B_3\ra_{\rm deg} A_5$,
see $6.56$. \\
The degenerations for $B_5(\be)$ are given by table $1$ and $6.57$, $6.58$.
\end{proof}

\begin{prop}\label{4.8}
The orbit closures for type $C\ra_{\rm deg}A$ are given as follows:
\vspace*{0.5cm}
\begin{center}
\begin{tabular}{c|c}
$A$ & $\partial (O(A))$ \\
\hline
$C_1$ & $A_1,\,A_2,\,A_5,\,A_6,\,A_9,\,A_{10}$    \\
$C_2$ & $A_1,\,A_5$   \\
$C_3$ & $A_1,\,A_5$    \\
$C_4$ & $A_1,\,A_5$    \\
$C_5(\al)$ & $A_1,\,A_5$\\
$C_6(\be)_{\be\neq -1}$ & $A_1,\,A_2,\,A_5,\,A_6,\,A_9,\,A_{10}$  \\
$C_6(-1)$ & $A_1,\,A_2,\,A_5$    \\
$C_7(\ga)_{\ga\neq 0}$ & $A_1,\,A_5$
\end{tabular}
\end{center}
\end{prop}
\vspace*{0.5cm}
\begin{proof}
The algebra $C_1$ satisfies $\det (L(x))=0$. Hence it cannot degenerate to $A_{11}$ or $A_{12}$, which
do not satisfy this identity. Furthermore $C_1$ cannot degenerate to $A_7$. To see this
we may assume that the matrix $g_t^{-1}=(f_{ij}(t))$ has lower-triangular form. A direct
computation of the orbit shows that $A\in \ov{O(C_1)}$ can only be commutative if 
$\lim_{t\to 0}f_{11}(t)= 0$. In this case all $A\in \ov{O(C_1)}$ are either complete or satisfy
$c_{1,1}(A)=1$. Since $A_7$ is not complete and satisfies $c_{1,1}(A_7)=2$, it cannot lie in the orbit
closure of $C_1$. On the hand $C_1$ degenerates to $A_6$, see $6.59$ and hence to $A_{10}$,
$A_2$, $A_9$, $A_5$ and $A_1$ by transitivity. By table $1$ there are no other posibilities. \\
The algebra $C_2$ can only degenerate to $A_{10}$, $A_2$, $A_9$, $A_{11}$, $A_5$ or $A_1$, see
table $1$. Since $c_{1,1}(C_2)=2$ and $c_{1,1}(A_2)=1$, $c_{1,1}(A_{11})=3$ we can exclude $A_2$ 
and $A_{11}$. $C_2$ cannot degenerate to $A_9$ by the same argument why $C_2$ cannot degenerate to
$B_3$, see proposition $\ref{4.6}$: all $A\in \ov{O(C_2)}$ either satisfy $\dim (A\cdot A)\ge 2$
or $\dim \CL(A)\ge 2$. However we have $\dim A_9\cdot A_9=\dim \CL(A_9)=1$. By transitivity,
$C_2$ cannot degenerate to $A_{10}$. On the other hand we have $C_2\ra_{\rm deg} B_5(0)\ra_{\rm deg}A_5$. \\
The algebra $C_3$ has the same possible candidates for degeneration as $C_2$, see table $1$.
Since $d_{1,1}(C_3)=2$ and $d_{1,1}(A_2)=1$, $d_{1,1}(A_{11})=3$ we can exclude $A_2$ and $A_{11}$. 
Since $\dim \CR(C_3)=2$ and $\dim \CR(A_9)=\dim \CR(A_{10})=1$, $C_3$ cannot degenerate to
$A_9$ or $A_{10}$ by lemma $\ref{inv4}$. By transitivity we have $C_3\ra_{\rm deg} \ra_{\rm deg} B_5(0)
\ra_{\rm deg}A_5$. \\
For $C_4$ it is again enough to exclude $A_{10}$, $A_2$, $A_9$, $A_{11}$, see table $1$.
Since $C_4$ is complete, $A_2$ and $A_{11}$ are not possible. Since $\dim \CL(C_3)=2$ and 
$\dim \CL(A_9)=\dim \CL(A_{10})=1$, $C_4$ cannot degenerate to $A_9$ or $A_{10}$ by lemma $\ref{inv4}$.
By transitivity we have $C_4\ra_{\rm deg} B_5(0)\ra_{\rm deg}A_5$. \\
For an algebra $C_5(\al)$ it is again anough to exclude $A_{10}$, $A_2$, $A_9$, $A_{11}$.
Since $C_5(0)$ is complete, it cannot degenerate to $A_2$ or $A_{11}$. Since $d_{1,1}(C_5(\al))=2$
for $\al\neq 0$, no algebra $C_5(\al)$ can degenerate to $A_2$ or $A_{11}$. 
We can also exclude $A_9$, and hence $A_{10}$ by considering $(0,1,\frac{-\al}{1+\al})$-derivations, 
for $\al\neq 0, \frac{-1}{2},-1$ as in the case of possible degenerations from $C_5(\al)$ to $B_3$,
see proposition $\ref{4.6}$. For $\al=0$ we use
$(0,1,0)$-derivations, and for  $\al=-1$ we use $(0,0,1)$-derivations. For $\al=\frac{1}{2}$
we note that $J_{C_5(-\frac{1}{2})}\simeq A_2$ and $J_{A_9}\simeq A_9$, so that we can exclude a 
degeneration from $C_5(-\frac{1}{2})$ to $A_9$ by lemma $\ref{inv5}$. By transitivity we have
$C_5(-\frac{1}{2})\ra_{\rm deg}B_3\ra_{\rm deg}A_5$ and $C_5(\al)\ra_{\rm deg}B_5(-\al)\ra_{\rm deg}A_5$
for $\al\neq -\frac{1}{2}$. \\
An algebra $C_6(\be)$ cannot degenerate to $A_7$ by the same argument why $C_1$ cannot degenerate
to $A_7$. The space $\Der_{(0,1,\frac{-\be}{\be+1})}(C_6(\be))$ is $1$-dimensional for $\be\neq 0,-1$
and $3$-dimensional for $\be=0$. Since $\dim \Der_{(0,1,\frac{-\be}{\be+1})}(A_{11})=0$ for $\be\neq -1$,
no algebra $C_6(\be)$ with $\be\neq -1$ can degenerate to $A_{11}$, and hence to $A_{12}$. 
On the other hand, $C_6(\be)$ degenerates to $A_6$ for $\be\neq -1$, see $6.60$, and 
hence to $A_2$, $A_{10}$, $A_9$, $A_5$ and $A_1$. For $\be=-1$, $C_6(-1)$ cannot degenerate to
$A_9$, $A_{10}$, $A_{11}$ and $A_{12}$ since $\dim \Der_{(1,1,0)}(C_6(-1))=5$ and the space of
$(1,1,0)$-derivations for $A_9,\ldots ,A_{12}$ is only $3$-dimensional. $C_6(-1)$ degenerates to $A_2$, 
see $6.61$, and hence to $A_5$ and $A_1$. \\
An algebra $C_7(\ga)$ cannot degenerate to $A_2$, since $d_{1,1}(C_7(\ga))=3$ and $d_{1,1}(A_2)=1$.
Because we have
\[
c_{i,j}(C_7(\ga))=\frac{(2\ga^i+(\ga+1)^i)(2\ga^j+(\ga+1)^j)}{2\ga^{i+j}+(\ga+1)^{i+j}},
\]
and $c_{i,j}(A_{11})=3$, there is no degeneration from  $C_7(\ga)$ to $A_{11}$. Suppose that
an algebra $C_7(\ga)$ degenerates to $A_9$. We may assume that this is realized via
$g_t^{-1}=(f_{ij}(t))$ having lower-triangular form. A direct computation of the orbit 
then shows that $A\in \ov{O(C_7(\ga))}$ can only be commutative if 
$\lim_{t\to 0}f_{11}(t)= 0$. In this case all $A\in \ov{O(C_7(\ga))}$ satisfy $\dim \CL(A)\ge 2$. Since
this is not true for $A_9$, the claim follows. By transitivity, $C_7(\ga)$ cannot degenerate 
to $A_{10}$ either. On the other hand $C_7(\ga)\ra_{\rm deg}B_5(0)\ra_{\rm deg} A_5$.
\end{proof}

\begin{prop}\label{4.9}
The orbit closures for type $E_{\la}\ra_{\rm deg}B$, $\la\neq 0,1$ are given as follows:
\vspace*{0.5cm}
\begin{center}
\begin{tabular}{c|c}
$A$ & $\partial (O(A))$ \\
\hline
$E_{1,\la}(\al)_{\al\neq -1,-\la}$ & $B_4(0),\,B_5(0)$  \\
$E_{1,\la}(-1)$ & $B_5(0)$  \\
$E_{1,\la}(-\la)\simeq E_{1,\frac{1}{\la}}(-1)$ & $B_5(0)$  \\
$E_{2,\la}$ & $B_4(0),\,B_5(0)$   \\
$E_3$ & $B_3,\,B_4(\al),\,B_5(\be)$    \\
$E_4$ & $B_3,\,B_4(\al),\,B_5(\be)$    \\
$E_5(\be)_{\be\neq -\frac{1}{2},-1}$ & $B_3,\,B_4(\al),\,B_5(\be)$  \\
$E_5( -\frac{1}{2})$ & $B_3,\,B_5(\be)$   \\
$E_5(-1)$ & $B_4(1),\,B_5(0)$   \\
$E_6$ & $B_3,\,B_5(\be)$
\end{tabular}
\end{center}
\end{prop}
\vspace*{0.5cm}
\begin{proof}
By table $1$, an algebra $E_{1,\la}(\al)$ can only degenerate to  $B_3$, $B_4(\ov{\al})$ or
$B_5(\ov{\be})$. For $\al=-1$ or $\la=-\al$ all algebras $B_4(\ov{\al})$ can be 
excluded. An algebra  $E_{1,\la}(\al)$ satisfies the identity 
$T(x)=(R(x)-\frac{1}{3}\tr (R(x))\id)^2=0$. Since $B_4(\al),\, \al\neq 0$
do not satisfy it, only $B_3$, $B_4(0)$ and $B_5(\ov{\be})$ remain. There is a degeneration
$E_{1,\la}(\al)\ra_{\rm deg} B_4(0)\ra_{\rm deg} B_5(0)$ for $(1-\la)(\al+1)(\al+\la)\neq 0$, 
see $6.72$. For $\al=-1$ we have $E_{1,\la}(\al)\ra_{\rm deg} B_5(0)$, see $6.73$. An algebra 
$E_{1,\la}(\al)$ cannot degenerate to $B_3$ or $B_5(\ov{\be})$ for $\ov{\be}\neq 0,1$. 
To see this we compute the orbit closure of $E_{1,\la}(\al)$ via a lower-triangular 
matrix $g_t^{-1}=(f_{ij}(t))$. Suppose that $B\in \ov{O(E_{1,\la}(\al))}$ satisfies 
$\dim (B\cdot B)=1$. It follows that $\lim_{t\to 0}f_{11}(t)= 0$, so that the 
only possibly nonzero products in $B$ are of the form $e_1\cdot e_1=\al_1e_2+\al_2e_3$,
$e_1\cdot e_2=\al_3e_3$. In particular, $\dim \CL(B)\ge 2$. This is not the case for
$B_3$ and $B_5(\be)$ with $\be\neq 0,1$. Hence only $B_5(0)\simeq B_5(1)$ lies in the orbit closure.\\
An algebra $E_{2,\la}$ can only degenerate to $B_4(0)$ and $B_5(0)$, and it does for $\la\neq 1$,
see $6.70$. The argument is the same as for $E_{1,\la}(\al)$. Also $E_{2,\la}$ satisfies $T(x)=0$. \\
The algebra $E_3$ cannot degenerate to $B_1$ since $c_{2,2}(E_3)=2$ and $c_{2,2}(B_1)=3$. By table
$1$, $E_3$ can only degenerate to $B_3$, $B_4(\al)$ or $B_5(\be)$. All cases are
possible, see $6.62$ and $6.63$ for $E_3\ra_{\rm deg}B_4(\al)\ra_{\rm deg}B_5(\al)$, and
$E_3\ra_{\rm deg}B_4(\frac{1}{2})\ra_{\rm deg}B_3$ by transitivity. \\
The algebra $E_4$ cannot degenerate to $B_1$ since $c_{1,1}(E_4)=\frac{9}{5}$ and $c_{1,1}(B_1)=3$.
On the other hand we have $E_4\ra_{\rm deg}B_4(\al)$ for $\al\neq 0,1$, see $6.65$,
$E_4\ra_{\rm deg}B_4(0)$, see $6.64$ and $E_4\ra_{\rm deg}E_5(-1)\ra_{\rm deg}B_4(1)$, see $6.66$. \\
An algebra $E_5(\be)$ cannot degenerate to $B_2$ by table $1$. It cannot degenerate to $B_1$
since $T_{\be,\frac{1}{2}}(x)=0$ holds for $E_5(\be)$ with $\be\neq 0$, see proposition $\ref{4.5}$,
but not for $B_1$. If $\be=0$, then $E_5(0)$ is complete, but $B_1$ is not. 
Every algebra $E_5(\be)$ with $\be\neq -1$ degenerates to every algebra $B_5(\ov{\be})$: 
for $\be\neq -\frac{1}{2},-1$ and $\ov{\be}\neq 0$ we have $E_5(\be)\ra_{\rm deg}B_4(\ov{\be})
\ra_{\rm deg}B_5(\ov{\be})$, see $6.66$, and for $\ov{\be}= 0$ we have
 $E_5(\be)\ra_{\rm deg}E_{1,\frac{1}{2}}(\be)\ra_{\rm deg}B_4(0)\ra_{\rm deg}B_5(0)$, see $6.72$.
For $\be=-\frac{1}{2}$ we have $E_5(-\frac{1}{2})\ra E_6\ra B_5(\ov{\be})$, see $6.68$, $6.69$.
Together this shows that we have a degeneration $E_5(\be)\ra_{\rm deg}B_5(\ov{\be})$ for $\be\neq -1$. 
It can also be given explicitly by 
\[
g_t^{-1}= 
\begin{pmatrix}
t^3 & 0 & 0 \\
\frac{\ov{\be}-1}{t} & t & 0 \\
\frac{(\ov{\be}-2-2\be)(\ov{\be}-1)}{(\be+1)t^5} & \frac{2}{t^3} & 1 
\end{pmatrix}. 
\]
The algebra $E_5(-1)$ degenerates to $B_5(\ov{\be})$ only for $\ov{\be}(\ov{\be}-1)=0$.
To see this we consider another basis $(e_1,e_2,e_3)$ for $E_5(-1)$, by interchanging
$e_2$ and $e_3$ from the old basis. Then the left multiplication operators have simultaneous
lower-triangular form. We compute the orbit closure of $E_5(-1)$ via a lower-triangular 
matrix $g_t^{-1}=(f_{ij}(t))$. Suppose that $B\in \ov{O(E_5(-1))}$ satisfies 
$\dim (B\cdot B)=1$. It follows that $\lim_{t\to 0}f_{11}(t)= 0$, and the 
products in $B$ are of the form $e_1\cdot e_1=\al_1e_3$, $e_1\cdot e_2=\al_2e_3$,
$e_2\cdot e_1=\al_3e_3$ and  $e_2\cdot e_2=\al_4e_3$, such that $\al_1\al_4=\al_2\al_3$.
Furthermore we assume that $B$ is not commutative, which implies that $\al_2-\al_3\neq 0$.
A direct computation now shows that such an algebra $B=B(\al_1,\al_2,\al_3,\al_4)$ is isomorphic
to $B_5(\ov{\be})$ if and only if $\ov{\be}(\ov{\be}-1)=0$. We have indeed $E_5(-1)\ra_{\rm deg}B_4(1)
\ra_{\rm deg}B_5(1)\simeq B_5(0)$, see $6.71$. It follows that $E_5(-1)$ cannot
degenerate to $B_3$, because otherwise $E_5(-1)\ra_{\rm deg}B_3\ra_{\rm deg}B_5(\frac{1}{2})$, which
we just excluded. On the other hand $E_5(\be)$ degenerates to $B_3$ for all $\be\neq -1$:
we have $E_5(\be)\ra_{\rm deg}B_4(\frac{1}{2})\ra_{\rm deg}B_3$ for $\be\neq -\frac{1}{2}-1$, 
see $6.66$, and $E_5(-\frac{1}{2})\ra_{\rm deg}E_6\ra_{\rm deg}B_3$, see $6.67$.
We already saw that $E_5(\be)$ degenerates to every $B_4(\al)$ for $\be\neq -\frac{1}{2},-1$.
On the other hand $E_5( -\frac{1}{2})$ does not degenerate to any $B_4(\al)$: compute the orbit
closure of $E_5(-\frac{1}{2})$ in the same way as for $E_5(-1)$ above. Assume that
$B\in \ov{O(E_5(-\frac{1}{2}))}$ satisfies $\dim (B\cdot B)<3$. It follows that 
$\lim_{t\to 0}f_{11}(t)= 0$. In this case however it is obvious that  $\dim (B\cdot B)\le 1$
for all algebras in the orbit closure. Since  $\dim B_4(\al)\cdot B_4(\al)=2$, the claim
follows. For $\al\neq 0$ we could have used another argument. Since $E_5(-\frac{1}{2})$
satisfies the operator identity $T(x)=(R(x)-\frac{1}{3}\tr (R(x))\id)^2=0$, but
$B_4(\al)$ for $\al\neq 0$ does not, the claim follows again for $\al\neq 0$. Suppose that algebra
$E_5(-1)$ degenerates to $B_4(\al)$. Then, by transitivity, it degenerates to $B_5(\al)$. As we have
seen, this is only possible for $\al=0,1$. In fact, there is a degeneration $E_5(-1)\ra_{\rm deg}B_4(1)$,
see $6.71$. There is no degeneration $E_5(-1)\ra_{\rm deg}B_4(0)$: suppose that $B\in \ov{O(E_5(-1))}$
with $\dim B\cdot B=2$. Computing the orbit via a lower-triangular matrix $g_t^{-1}=(f_{ij}(t))$
as above, we see that $\lim_{t\to 0}f_{11}(t)= 0$ and the products in $B$ are of the form 
$e_1\cdot e_1=\al_1e_2+\al_2e_3$, $e_1\cdot e_2=\al_3e_3$, $e_2\cdot e_1=\al_4e_3$ and  
$e_2\cdot e_2=\al_5e_3$ with $\al_1\neq 0$. If $B$ satisfies the operator identity $T(x)=0$ as 
$B_4(0)$ does, then $\al_4=\al_5=0$. In this case we have $\dim \CL(B)\ge 2$, whereas 
$\dim \CL(B_4(0))=1$. This shows the claim. \\
By table $1$, the algebra $E_6$ can only degenerate to $B_3$, $B_4(\al)$ or $B_5(\be)$.
By transitivity $B_4(\al)$ is not possible, because otherwise we would have a degeneration
$E_5(-\frac{1}{2})\ra_{\rm deg}E_6\ra_{\rm deg}B_4(\al)$, which is impossible. There are degenerations
$E_6\ra_{\rm deg}B_3$ and $E_6\ra_{\rm deg}B_5(\be)$, see $6.67$ and $6.68$, $6.69$.
\end{proof}

\begin{prop}
The orbit closures for type $E_{\la}\ra_{\rm deg}A$, $\la\neq 0$ are given as follows:
\vspace*{0.5cm}
\begin{center}
\begin{tabular}{c|c}
$A$ & $\partial (O(A))$ \\
\hline
$E_{1,\la}(\al)_{\al\neq -1}$ & $A_1,\,A_5$  \\
$E_{1,\la}(-1),\, \la\neq 1$ & $A_1,\,A_5$  \\
$E_{1,1}(-1)$ & $A_1$  \\
$E_{2,\la}$ & $A_1,\,A_5$   \\
$E_3$ & $A_1,\,A_5,\,A_9,\,A_{10}$    \\
$E_4$ & $A_1,\,A_5,\,A_9,\,A_{10}$    \\
$E_5(\be)_{\be\neq -\frac{1}{2},-1}$ & $A_1,\,A_5,\,A_9,\,A_{10}$  \\
$E_5( -\frac{1}{2})$ & $A_1,\,A_5,\,A_9$   \\
$E_5(-1)$ & $A_1,\,A_5$   \\
$E_6$ & $A_1,\,A_5,\,A_9$
\end{tabular}
\end{center}
\end{prop}
\vspace*{0.5cm}
\begin{proof}
By table $1$, an algebra $E_{1,\la}(\al)$ can only degenerate to $A_{10}$, $A_2$, $A_9$,
$A_{11}$, $A_5$ or $A_1$. We can exclude $A_2$ and $A_{11}$ because $c_{i,j}(A_2)=1$ and 
$c_{i,j}(A_{11})=3$ for all $i,j\ge 1$, whereas 

\[
c_{i,j}(E_{1,\la}(\al))=\frac{(\al^i+(\al+1)^i+(\al+\la)^i)(\al^j+(\al+1)^j+(\al+\la)^j)}
{\al^{i+j}+(\al+1)^{i+j}+(\al+\la)^{i+j}}.
\]
Furthermore we can exclude $A_9$ and hence $A_{10}$ by transitivity, because
$\dim \Der_{(0,1,0)}(A_9)=3$ and $\dim \Der_{(0,1,0)}(E_{1,\la}(\al))=6$.
For $\la\neq 1$ we have $E_{1,\la}(\al)\ra_{\rm deg} B_5(0)\ra_{\rm deg} A_5\ra_{\rm deg} A_1$.
For $\la=1$ and $\al\neq -1$ we have $E_{1,1}(\al)\ra_{\rm deg} A_5$, see $6.78$. By table $1$,
$E_{1,1}(-1)$ can only degenerate to $A_1$, which is of course possible, see $6.80$. \\
By table $1$, $E_{2,\la}$ can only degenerate to  $A_{10}$, $A_2$, $A_9$, $A_{11}$, $A_5$ or $A_1$.
Because $c_{i,j}(E_{2,\la})=c_{i,j}(E_{1,\la}(-1))$, we can exclude $A_2$ and $A_{11}$.
We can also exclude $A_9$, and hence $A_{10}$ by the same argument why $C_2$ cannot degenerate to
$A_9$, see proposition $\ref{4.8}$. For $\la\neq 1$ we have $E_{2,\la}\ra_{\rm deg} B_5(0)\ra_{\rm deg} 
A_5\ra_{\rm deg} A_1$. For $\la=1$ we have $E_{2,1}\ra_{\rm deg} A_5$, see $6.79$.\\
The algebra $E_3$ cannot degenerate to $A_2$, $A_6$, $A_7$, $A_{11}$, $A_{12}$ because of
$c_{i,j}(A_2)=c_{i,j}(A_6)=1$, $c_{i,j}(A_7)=2$, $c_{i,j}(A_{11})=c_{i,j}(A_{12})=3$ and
\[
c_{i,j}(E_3)=\frac{((-1)^i+1)((-1)^j+1)}{(-1)^{i+j}+1}.
\]
By table $1$, we can exclude $A_3$, $A_8$ and $A_6$. There are degenerations
$E_3\ra_{\rm deg} A_{10}\ra_{\rm deg} A_9\ra_{\rm deg} A_5\ra_{\rm deg} A_1$, see $6.74$.\\
The algebra $E_4$ also degenerates to $A_{10}$, see $6.75$, and hence to $A_9$, $A_5$ and $A_1$.
The other algebras are excluded by table $1$ and the $c_{i,j}$-invariants: we have
\[
c_{i,j}(E_4)=\frac{(2^i+1)(2^j+1)}{2^{i+j}+1}.
\] 
An algebra $E_5(\be)$ can only degenerate $A_{10}$, $A_9$, $A_5$ and $A_1$. This follows from 
table $1$ and the $c_{i,j}$-invariants. We have
\[
c_{i,j}(E_5(\be))=\frac{(\be^i+(\be+1)^i+(\be+\frac{1}{2})^i)(\be^j+(\be+1)^j+(\be+\frac{1}{2})^j)}
{\be^{i+j}+(\be+1)^{i+j}+(\be+\frac{1}{2})^{i+j}}.
\]
For $\be\neq -\frac{1}{2},-1$ there is a degeneration $E_5(\be)\ra_{\rm deg}A_{10}$, see $6.76$, so
that $A_{10}$, $A_9$, $A_5$ and $A_1$ are in the orbit closure. For $\be= -\frac{1}{2}$, the algebra
$E_5( -\frac{1}{2})$ satisfies  $(R(x)-\frac{1}{3}\tr (R(x))\id)^2=0$, but $A_{10}$ does not.
Hence $E_5( -\frac{1}{2})$ does not degenerate to $A_{10}$. It degenerates however to $A_9$, $A_5$ and
$A_1$ by transitivity: $E_5( -\frac{1}{2})\ra_{\rm deg}E_6\ra_{\rm deg} A_9$, see $6.77$.
For $\be=-1$, $E_5(-1)$ does not degenerate to $A_9$: we use the same proof as in proposition
$\ref{4.9}$ which showed that $E_5(-1)$ does not degenerate to an algebra $B_5(\be)$ with
$\be(1-\be)\neq 0$. An algebra $A \in \ov{O(E_5(-1))}$ satisfying $\dim (A\cdot A)=1$ can be represented
by products $e_1\cdot e_1=\al_1e_3$, $e_1\cdot e_2=\al_2e_3$,
$e_2\cdot e_1=\al_3e_3$ and  $e_2\cdot e_2=\al_4e_3$, such that $\al_1\al_4=\al_2\al_3$. If we now
assume that $A$ is commutative, it follows $\al_3=\al_2$ and $\al_1\al_4=\al_2^2$. It is easy
to see that such an algebra cannot be isomorphic to $A_9$. There is a degeneration
$E_5(-1)\ra_{\rm deg}B_5(0)\ra_{\rm deg}A_5$. \\
By table $1$, $E_6$ can only degenerate to $A_{10}$, $A_2$, $A_9$, $A_{11}$, $A_5$ and $A_1$. 
We can exclude $A_2$ and $A_{11}$ as before since $c_{i,j}(E_6)=c_{i,j}(E_3)$. Furthermore
$E_6$ satisfies $(R(x)-\frac{1}{3}\tr (R(x))\id)^2=0$, but $A_{10}$ does not. Hence we may exclude
$A_{10}$. There is a degeneration $E_6\ra_{\rm deg}A_9$, see $6.77$.
\end{proof}

\begin{prop}\label{4.11}
The orbit closures for type $D\ra_{\rm deg}B$ are given as follows:
\vspace*{0.5cm}
\begin{center}
\begin{tabular}{c|c}
$A$ & $\partial (O(A))$ \\
\hline
$D_1$ & $B_4(0),\, B_5(0)$    \\
$D_2(\al)_{\al\neq -1}$ & $B_4(0),\, B_5(0)$  \\
$D_2(-1)$ & $B_5(0)$  
\end{tabular}
\end{center}
\end{prop}
\vspace*{0.5cm}

\begin{proof}
By table $1$, $D_1$ can only degenerate to $B_3,B_4(\al)$ or $B_5(\be)$. Since $D_1$ 
satisfies the operator identity $(R(x)-\frac{1}{3}\tr (R(x))\id)^2=0$, but $B_4(\al)$ 
for $\al\neq 0$ does not, we can exclude $B_4(\al)$ for $\al\neq 0$. There is 
a degeneration $D_1\ra_{\rm deg} B_4(0)\ra_{\rm deg}B_5(0)$, see $6.81$. On the other hand,
$D_1$ does not degenerate to $B_3$ or an algebra $B_5(\be)$ with $\be(1-\be)\neq 0$. 
To see this, change the basis of $D_1$ by permuting $e_2$ and $e_3$, and compute the orbit 
closure of $D_1$ via a lower-triangular matrix $g_t^{-1}=(f_{ij}(t))$. Suppose that 
$B\in \ov{O(D_1)}$ satisfies $\dim (B\cdot B)=1$. It follows that $\lim_{t\to 0}f_{11}(t)= 0$ 
and $\dim \CL(B)\ge 2$. Hence $B$ cannot be isomorphic to $B_3$ or $B_5(\be)$ with 
$\be(1-\be)\neq 0$, see proposition $\ref{4.2}$. \\
The same arguments show that an algebra $D_2(\al)$ with $\al\neq -1$ can only degenerate to
$B_4(0)$ and $B_5(0)$. There is a degeneration $D_2(\al)\ra_{\rm deg} B_4(0)$ for $\al\neq -1$, 
see $6.82$. By table $1$, the algebra $D_2(-1)$ can only degenerate to $B_3$ or $B_5(\be)$.
Again the above argument shows that only $B_5(0)\simeq B_5(1)$ remains. There is a degeneration
$D_2(-1)\ra_{\rm deg}B_5(0)$, see $6.83$.
\end{proof}

\newpage

\begin{prop}\label{4.12}
The orbit closures for type $D\ra_{\rm deg}E_{\la=1}$ are given as follows:
\vspace*{0.5cm}
\begin{center}
\begin{tabular}{c|c}
$A$ & $\partial (O(A))$ \\
\hline
$D_1$ & $E_{1,1}(-1),\,E_{2,1}$    \\
$D_2(\al)_{\al\neq -1}$ & $E_{1,1}(\al)_{\al\neq -1}$  \\
$D_2(-1)$ & $E_{1,1}(-1),\,E_{2,1}$ 
\end{tabular}
\end{center}
\end{prop}
\vspace*{0.5cm}

\begin{proof}
The algebra $D_1$ satisfies the operator identity $T(x)=L(x)^3-L(x)^2R(x)=0$. This identity
holds for $E_{1,1}(\al)$ if and only if $\al=-1$. Hence $D_1$ cannot degenerate to $E_{1,1}(\al)$
for $\al\neq -1$. There are degenerations $D_1\ra_{\rm deg} D_2(-1)\ra_{\rm deg} E_{2,1}\ra_{\rm deg}
E_{1,1}(-1)$, see $6.85$. \\
A degeneration $D_2(\al)\ra_{\rm deg}E_{2,1}$ can only exist for $\al=-1$, since $D_2(\al)$ satisfies
the operator identity $T_{\al}(x)=0$ with
\begin{align*}
T_{\al}(x) & = \al^2L(x)^3-\al^2L(x)^2R(x)-2\al(\al+1)R(x)L(x)^2\\
           & \quad +2\al(\al+1)R(x)L(x)R(x) + (\al+1)^2R(x)^2L(x)-(\al+1)^2R(x)^3,
\end{align*}
whereas $E_{2,1}$ satisfies it if and only if $\al=-1$. There is a degeneration
$D_2(\al)\ra_{\rm deg}E_{1,1}(\al)$ for every $\al$, see $6.84$.
\end{proof}

\begin{prop}
The orbit closures for type $D\ra_{\rm deg}A$ are given as follows:
\vspace*{0.5cm}
\begin{center}
\begin{tabular}{c|c}
$A$ & $\partial (O(A))$ \\
\hline
$D_1$ & $A_1,\,A_5$    \\
$D_2(\al)$ & $A_1,\,A_5$
\end{tabular}
\end{center}
\end{prop}
\vspace*{0.5cm}

\begin{proof}
By table $1$, $D_1$ can only degenerate to $A_{10}$, $A_2$, $A_9$, $A_{11}$, $A_5$ or $A_1$.
Since $c_{1,1}(D_1)=1$, $d_{1,1}(D_1)=3$, but $c_{1,1}(A_{11})=3$, $d_{1,1}(A_2)=1$, we can
exclude $A_2$ and $A_{11}$. Furthermore $D_1$ cannot degenerate to $A_9$, and hence not
to $A_{10}$ by transitivity. To see this, we use the argument of proposition $\ref{4.11}$.
A commutative algebra $A\in \ov{O(D_1)}$ with $\dim (A\cdot A)=1$ can be represented by
products of the form $e_1\cdot e_1=\al e_3$. It cannot be isomorphic to $A_9$. There is
a degeneration $D_1\ra_{\rm deg}B_5(0)\ra_{\rm deg} A_5$. \\
For an algebra $D_2(\al)$, exactly the same arguments yield the claim. There is a degeneration
$D_2(\al)\ra_{\rm deg} B_5(0)\ra_{\rm deg} A_5$.
\end{proof}

\newpage

\section{Hasse diagrams}

1.) Type $A\ra_{\rm deg}A$:
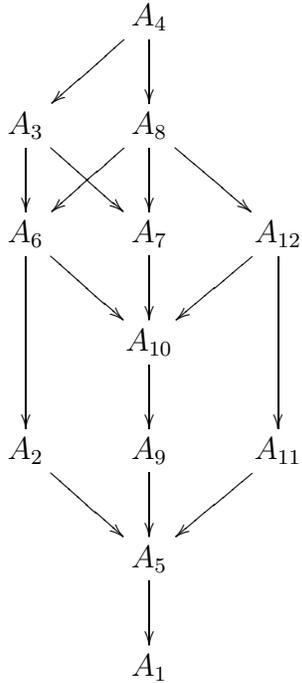
$$
\begin{xy}
\xymatrix{
 & A_4 \ar[ld] \ar[d] & \\
A_3\ar[d]\ar[rd]  &  A_8\ar[ld]\ar[d]\ar[rd]  &  \\
A_6\ar[dd] \ar[rd] & A_7\ar[d] &  A_{12}\ar[ld]\ar[dd] \\
 & A_{10}\ar[d] & \\
A_2\ar[rd] & A_9 \ar[d] & A_{11}\ar[ld]\\
 & A_5\ar[d] & \\
 & A_1 & 
}
\end{xy}
$$
\vspace*{0.5cm}
2.) Type $B\ra_{\rm deg}B$: 
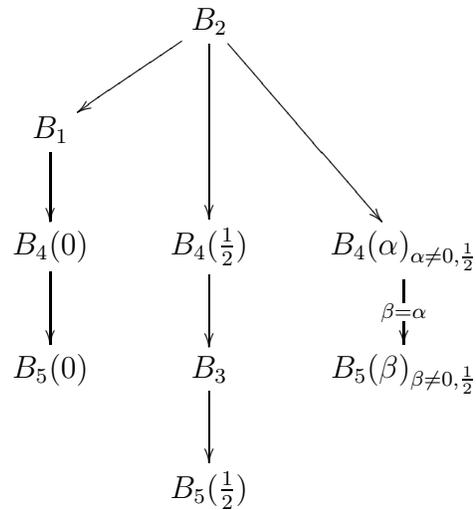
$$
\begin{xy}
\xymatrix{
 & B_2 \ar[ld] \ar[dd] \ar[rdd] & \\
B_1\ar[d]  &   &  \\
B_4(0)\ar[d] & B_4(\frac{1}{2})\ar[d] &  B_4(\al)_{\al\neq 0,\frac{1}{2}}\ar[d]|-{\be=\al} \\
B_5(0) & B_3\ar[d] &  B_5(\be)_{\be\neq 0,\frac{1}{2}}\\
 & B_5(\frac{1}{2}) & 
}
\end{xy}
$$
\vspace*{0.5cm}
3.) Type $C\ra_{\rm deg}C$: 
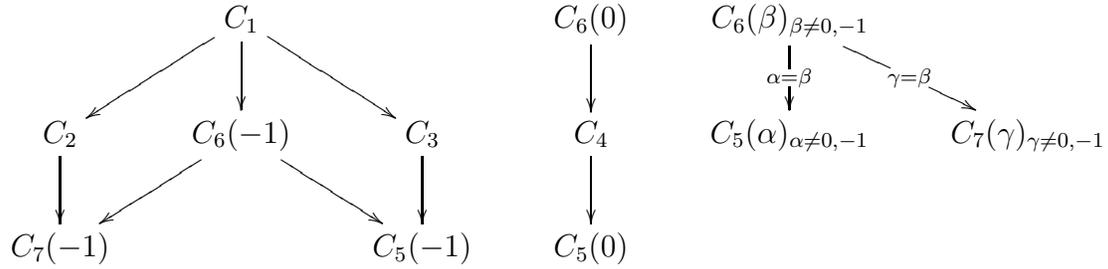
$$
\begin{xy}
\xymatrix{
 & C_1 \ar[ld] \ar[d] \ar[rd] & & C_6(0)\ar[d] & C_6(\be)_{\be\neq 0,-1} \ar[d]|-{\al = \be}
 \ar[rd]|-{\ga = \be}\\
C_2\ar[d]& C_6(-1)\ar[ld]\ar[rd] & C_3 \ar[d] & C_4 \ar[d] & C_5(\al)_{\al\neq 0,-1}
& C_7(\ga)_{\ga\neq 0,-1}\\
C_7(-1) &  &  C_5(-1) & C_5(0) & & 
}
\end{xy}
$$
\vspace*{0.5cm}
4.) Type $D\ra_{\rm deg}D$: 
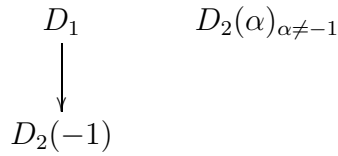
$$
\begin{xy}
\xymatrix{
D_1 \ar[d] & D_2(\al)_{\al\neq -1}\\
D_2(-1) &  
}
\end{xy}
$$
\vspace*{0.5cm}
5.) Type $E_\la\ra_{\rm deg}E_\la$ with $\la\neq 0$: 
$$
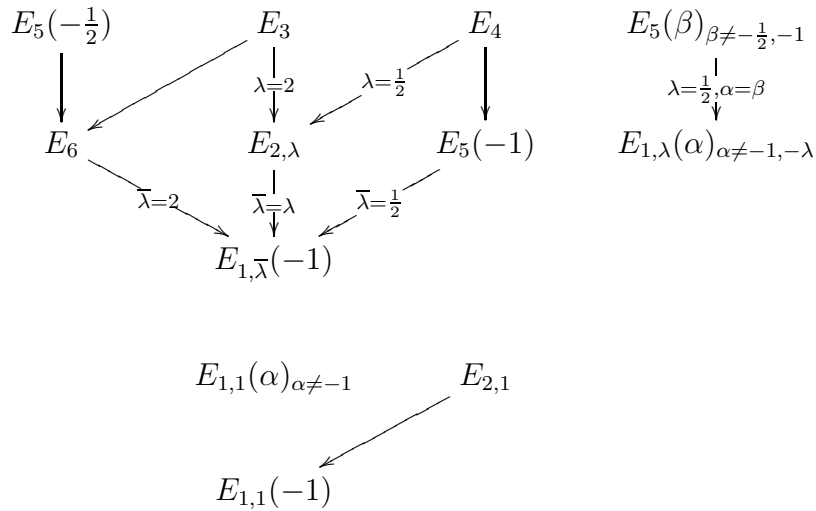
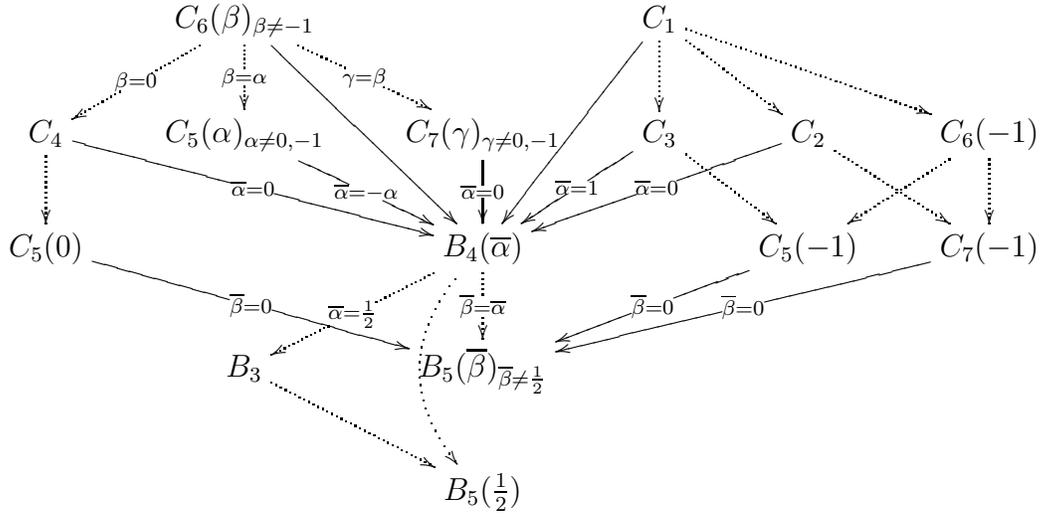
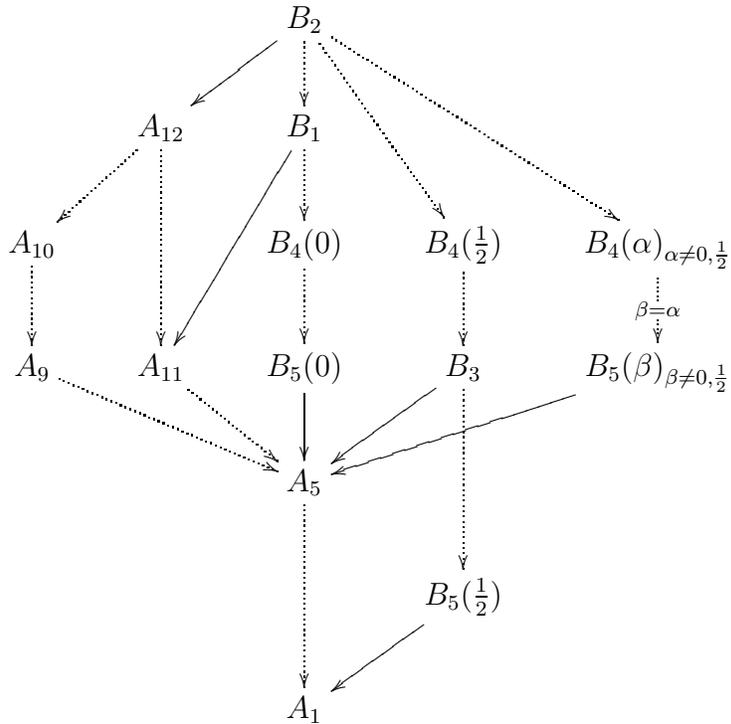
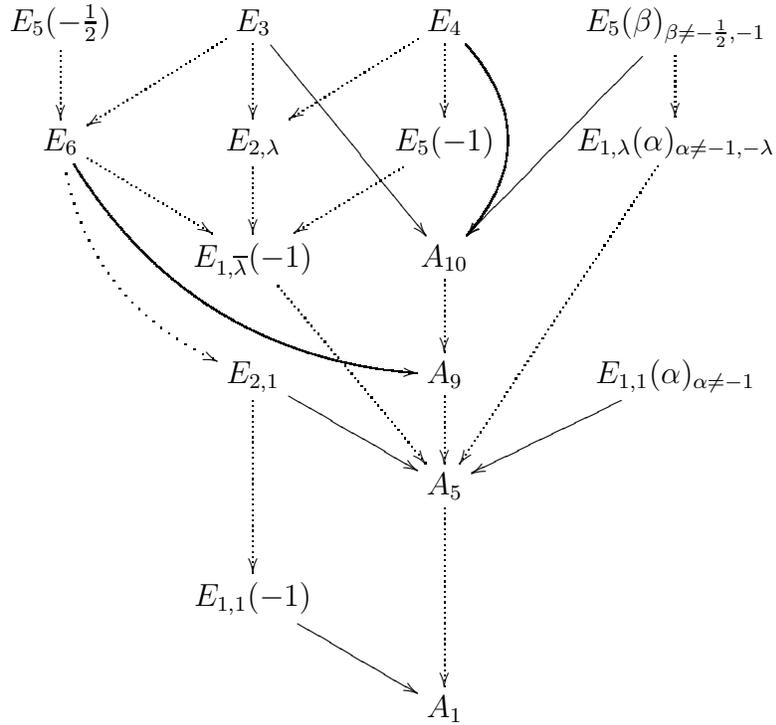
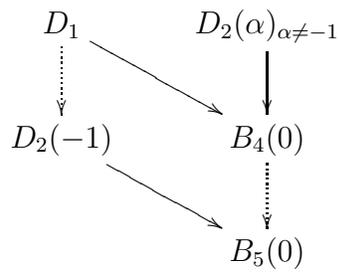
\begin{xy}
\xymatrix{
E_5(-\frac{1}{2})\ar[d] & E_3 \ar[ld] \ar[d]|-{\la = 2} & E_4 \ar[d] \ar[ld]|-{\la=\frac{1}{2}} & 
E_5(\be)_{\be\neq -\frac{1}{2},-1} \ar[d]|-{\la = \frac{1}{2},\al=\be} \\
E_6\ar[rd]|-{\ov{\la}=2} & E_{2,\la}\ar[d]|-{\ov{\la}=\la} & E_5(-1) \ar[ld]|-{\ov{\la}=\frac{1}{2}} 
 & E_{1,\la}(\al)_{\al\neq -1,-\la} \\
 &  E_{1,\ov{\la}}(-1) &  & \\  
 & E_{1,1}(\al)_{\al\neq -1} & E_{2,1}\ar[ld] & \\
 & E_{1,1}(-1)
}
\end{xy}
$$
\vspace*{0.5cm}
6.) Type $C\ra_{\rm deg}B$: 
$$
\begin{xy}
\xymatrix{
& C_6(\be)_{\be\neq -1} \ar @{.>}[ld]|-{\be=0} \ar @{.>}[d]|-{\be=\al} \ar @{.>}[rd]|-{\ga = \be} 
\ar[rdd] & &  
C_1\ar @{.>}[d] \ar[ldd] \ar @{.>}[rd]\ar @{.>}[rrd] & & \\
C_4 \ar @{.>}[d]  \ar[rrd]|-{\ov{\al}=0} & C_5(\al)_{\al\neq 0,-1} \ar[rd]|-{\ov{\al}=-\al} 
& C_7(\ga)_{\ga\neq 0,-1} \ar[d]|-{\ov{\al}=0} & C_3 \ar[ld]|-{\ov{\al}=1} \ar @{.>}[rd] 
& C_2  \ar[lld]|-{\ov{\al}=0}\ar @{.>}[rd] 
& C_6(-1) \ar @{.>}[ld] \ar @{.>}[d]\\
C_5(0) \ar[rrd]|-{\ov{\be}=0} & & B_4(\ov{\al})\ar @{.>}[ld]|-{\ov{\al}=\frac{1}{2}}
\ar @{.>} @/_2.0pc/[dd] \ar @{.>}[d]|-{\ov{\be}=\ov{\al}} & & 
C_5(-1)\ar[lld]|-{\ov{\be}=0} & C_7(-1)\ar[llld]|-{\ov{\be}=0} \\
       & B_3 \ar @{.>}[rd] & B_5(\ov{\be})_{\ov{\be}\neq \frac{1}{2}} & & & \\
       &                   & B_5(\frac{1}{2}) & & &
}
\end{xy}
$$
\vspace*{0.5cm}
7.) Type $B\ra_{\rm deg}A$: 
$$
\begin{xy}
\xymatrix{
& & B_2 \ar[ld] \ar @{.>}[d] \ar @{.>}[rdd] \ar @{.>}[rrdd] & &  \\
& A_{12}\ar @{.>}[ld] \ar @{.>}[dd] & B_1 \ar[ldd] \ar @{.>}[d] & & \\
A_{10} \ar @{.>}[d] & & B_4(0) \ar @{.>}[d] & B_4(\frac{1}{2}) \ar @{.>}[d] 
& B_4(\al)_{\al\neq 0,\frac{1}{2}} \ar @{.>}[d]|-{\be=\al} \\
A_9 \ar @{.>}[rrd] & A_{11} \ar @{.>}[rd] & B_5(0) \ar[d] & B_3\ar[ld] \ar @{.>}[dd] 
& B_5(\be)_{\be\neq 0,\frac{1}{2}} \ar[lld] \\
& & A_5\ar @{.>}[dd] & & \\
& & & B_5(\frac{1}{2})\ar[ld] & \\
& & A_1 & &  
}
\end{xy}
$$
\vspace*{0.5cm}
8.) Type $C\ra_{\rm deg}A$: 
$$
\begin{xy}
\xymatrix{
& C_1 \ar @{.>}[ld] \ar @{.>}[d] \ar @{.>}[rd] \ar[rrd] & & C_6(0) \ar[d]\ar @{.>}[rd] & & 
C_6(\be)_{\be\neq 0,-1} \ar[lld] \ar @{.>}[d]|-{\al=\be} \ar @{.>}[rd]|-{\ga=\be} &  \\
C_2 \ar @{.>}[d] & C_6(-1) \ar @{.>}[ld] \ar[rdd]\ar @{.>}[rd] & C_3 \ar @{.>}[d] &
A_6 \ar @{.>}[ldd] \ar @{.>}[d] & C_4 \ar @{.>}[d] &  C_5(\al)_{\al\neq 0,-1} \ar @{.>}[llddd] &
C_7(\ga)_{\ga\neq 0,-1} \ar @{.>}[lllddd] \\
C_7(-1)\ar @{.>}[rrrdd] & & C_5(-1)\ar @{.>}[rdd] & A_{10} \ar @{.>}[d] & C_5(0)\ar @{.>}[ldd] & & \\
& & A_2 \ar @{.>}[rd] & A_9 \ar @{.>}[d] & & & \\ 
& & & A_5 \ar @{.>}[d] & & & \\
& & & A_1 & & &  
}
\end{xy}
$$
\vspace*{0.5cm}
9.) Type $E_{\la\neq 1}\ra_{\rm deg}B$: 
$$
\begin{xy}
\xymatrix{
E_5(-\frac{1}{2})\ar @{.>}[d] & E_3 \ar @{.>}[ld] \ar @{.>}[d]\ar[rdd] 
& E_4 \ar @{.>}[d] \ar @{.>}[ld] \ar @/^2pc/[dd] & 
E_5(\be)_{\be\neq -\frac{1}{2},-1}\ar[ldd] \ar @{.>}[d] \\
E_6\ar[dd] \ar[rdd]\ar @{.>}[rd] & E_{2,\la}\ar @{.>}[d]\ar[rd]|-{\ov{\al}=0} & 
E_5(-1) \ar[d]|-{\ov{\al}=1}\ar @{.>} @/^0.8pc/[ld] & E_{1,\la}(\al)_{\al\neq -1,-\la}
\ar[ld]|-{\ov{\al}=0} \\
 &  E_{1,\ov{\la}}(-1)\ar[d]|-{\be = 0} & B_4(\ov{\al}) \ar @{.>}[ld] \ar @{.>} @/_1.0pc/[lld] & \\
B_3 \ar @{.>}[d] & B_5(\ov{\be})_{\ov{\be}\neq \frac{1}{2}} & & \\
B_5(\frac{1}{2}) & & &  
}
\end{xy}
$$
\vspace*{0.5cm}
10.) Type $E_{\la}\ra_{\rm deg}A$: 
$$
\begin{xy}
\xymatrix{
E_5(-\frac{1}{2})\ar @{.>}[d] & E_3 \ar @{.>}[ld] \ar @{.>}[d]\ar[rdd] 
& E_4 \ar @{.>}[d] \ar @{.>}[ld] \ar @/^2pc/[dd] & 
E_5(\be)_{\be\neq -\frac{1}{2},-1}\ar[ldd] \ar @{.>}[d] \\
E_6 \ar @{.>} @/_1.5pc/[rdd]\ar @{.>}[rd] \ar @/_2pc/[rrdd] & E_{2,\la}\ar @{.>}[d] & 
E_5(-1) \ar @{.>}[ld] & E_{1,\la}(\al)_{\al\neq -1,-\la} \ar @{.>}[lddd]\\
 &  E_{1,\ov{\la}}(-1) \ar @{.>}[rdd] & A_{10} \ar @{.>}[d] & \\
& E_{2,1}\ar @{.>}[dd] \ar[rd] & A_9 \ar @{.>}[d]& E_{1,1}(\al)_{\al\neq -1} \ar[ld] \\
& & A_5 \ar @{.>}[dd] & \\
& E_{1,1}(-1)\ar[rd] & &  \\
& & A_1 & 
}
\end{xy}
$$
\vspace*{0.5cm}
11.) Type $D\ra_{\rm deg}B$: 
$$
\begin{xy}
\xymatrix{
D_1\ar @{.>}[d] \ar[rd] & D_2(\al)_{\al\neq -1} \ar[d] \\
D_2(-1) \ar[rd] & B_4(0)\ar @{.>}[d] \\
 & B_5(0) 
}
\end{xy}
$$
\vspace*{0.5cm}
12.) Type $D\ra_{\rm deg}E_{\la=1}$:  
$$
\begin{xy}
\xymatrix{
D_1\ar @{.>}[d]  & D_2(\al)_{\al\neq -1} \ar[ddd] \\
D_2(-1) \ar[d] &  \\
E_{2,1} \ar @{.>}[dd] & \\
 & E_{1,1}(\al)_{\al\neq -1} \\
E_{1,1}(-1) &  
}
\end{xy}
$$
\vspace*{0.5cm}
13.) Type $D\ra_{\rm deg}A$:
$$
\begin{xy}
\xymatrix{
D_1\ar @{.>}[d]  & D_2(\al)_{\al\neq -1} \ar @{.>}[ldd] \\
D_2(-1) \ar @{.>}[d] &  \\
A_5 \ar @{.>}[d] & \\
A_1 &  
}
\end{xy}
$$

\section{List of essential degenerations}

\vspace*{0.5cm}
\begin{center}
\begin{tabular}{c|c|c|c|c|c|c}
               & $1$ & $2$ & $3$ & $4$ & $5$ & $6$ \\
\hline
& & & & & & \\
$\ra_{\rm deg}$ & $A_4\ra A_3$ & $A_4\ra A_8$ & $A_3\ra A_6$ & $A_3\ra A_7$ & $A_8\ra A_6$ & $A_8\ra A_7$ \\
& & & & & & \\
$g_t^{-1}$ & $\left(\begin{smallmatrix}
t & 0 & 0 \\
0 & 1 & 0 \\
0 & 0 & 1 
\end{smallmatrix}\right)$ &
$\left(\begin{smallmatrix}
t & 1 & 0 \\
0 & 1 & 0 \\
0 & 0 & 1 
\end{smallmatrix}\right)$ &
$\left(\begin{smallmatrix}
1 & -t^{-1} & 0 \\
0 & t & t^3 \\
0 & 0 & 1 
\end{smallmatrix}\right)$ &
$\left(\begin{smallmatrix}
0 & 0 & 1 \\
t & 1 & t \\
0 & 1 & t 
\end{smallmatrix}\right)$ &
$\left(\begin{smallmatrix}
1 & t^{-1} & 1 \\
0 & t & t^2 \\
0 & 0 & 1 
\end{smallmatrix}\right)$ &
$\left(\begin{smallmatrix}
1 & 0 & 0 \\
0 & 1 & 0 \\
0 & 0 & t 
\end{smallmatrix}\right)$ \\
\end{tabular}
\end{center}
\vspace*{0.5cm}
\begin{center}
\begin{tabular}{c|c|c|c|c|c|c}
               & $7$ & $8$ & $9$ & $10$ & $11$ & $12$ \\
\hline
& & & & & & \\
$\ra_{\rm deg}$ & $A_8\ra A_{12}$ & $A_6\ra A_2$ & $A_6\ra A_{10}$ & $A_7\ra A_{10}$ & $A_{12}\ra A_{10}$ 
& $A_{12}\ra A_{11}$ \\
& & & & & & \\
$g_t^{-1}$ & 
$\left(\begin{smallmatrix}
1 & t^{-1} & 0 \\
0 & t & 1 \\
0 & 0 & 1 
\end{smallmatrix}\right)$ &
$\left(\begin{smallmatrix}
1 & 0 & 0 \\
0 & t & 0 \\
0 & 0 & 1 
\end{smallmatrix}\right)$ &
$\left(\begin{smallmatrix}
-t & 1 & 0 \\
0 & t & -1 \\
0 & 0 & t 
\end{smallmatrix}\right)$ &
$\left(\begin{smallmatrix}
t^2 & 2t & 1 \\
0 & t^2 & t \\
0 & 0 & 1 
\end{smallmatrix}\right)$ &
 $\left(\begin{smallmatrix}
t & 1 & 0 \\
0 & t & 1 \\
0 & 0 & t 
\end{smallmatrix}\right)$ &
$\left(\begin{smallmatrix}
1 & 0 & 0 \\
0 & t & 0 \\
0 & 0 & 1 
\end{smallmatrix}\right)$
\end{tabular}
\end{center}
\vspace*{0.5cm}
\begin{center}
\begin{tabular}{c|c|c|c|c|c|c}
               & $13$ & $14$ & $15$ & $16$ & $17$ & $18$ \\
\hline
& & & & & & \\
$\ra_{\rm deg}$ & $A_{10}\ra A_9$ & $A_2\ra A_5$ & $A_9\ra A_5$ & $A_{11}\ra A_5$ & $A_5\ra A_1$ 
& $B_2\ra B_1$ \\
& & & & & & \\
$g_t^{-1}$ & 
$\left(\begin{smallmatrix}
t & 0 & 0 \\
0 & 1 & 0 \\
0 & 0 & t 
\end{smallmatrix}\right)$ &
$\left(\begin{smallmatrix}
0 & 0 & 1 \\
t & -1 & 0 \\
0 & t & 0 
\end{smallmatrix}\right)$ &
$\left(\begin{smallmatrix}
2 & 0 & 0 \\
0 & 1 & t \\
0 & 1 & 0 
\end{smallmatrix}\right)$ &
$\left(\begin{smallmatrix}
0 & 0 & 1 \\
t & 1 & 0 \\
0 & t & 0 
\end{smallmatrix}\right)$ &
$\left(\begin{smallmatrix}
t & 0 & 0 \\
0 & t & 0 \\
0 & 0 & t 
\end{smallmatrix}\right)$ &
$\left(\begin{smallmatrix}
1 & 0 & 0 \\
0 & t & 0 \\
0 & 0 & t 
\end{smallmatrix}\right)$
\end{tabular}
\end{center}
\vspace*{0.5cm}
\begin{center}
\begin{tabular}{c|c|c|c|c|c}
               & $19$ & $20$ & $21$ & $22$ & $23$ \\
\hline
& & & & & \\
$\ra_{\rm deg}$ & $B_2\ra B_4(\al)$ & $B_1\ra B_4(0)$ & $B_4(\al)\ra B_5(\al)$ & 
$B_4(\frac{1}{2})\ra B_3$ & $B_3\ra B_5(\frac{1}{2})$ \\
& 
$\begin{smallmatrix}
\al\neq 0 \\
\end{smallmatrix}$
& & & & \\
$g_t^{-1}$ & 
$\left(\begin{smallmatrix}
0 & t & 0 \\
-\al t^2 & -\al t & 0 \\
0 & \al(1-\al)t & \al t^3 
\end{smallmatrix}\right)$ &
$\left(\begin{smallmatrix}
0 & t & 0 \\
1 & t^{-1} & 0 \\
0 & -t^{-1} & -t 
\end{smallmatrix}\right)$ &
$\left(\begin{smallmatrix}
1 & 0 & 0 \\
0 & t & 0 \\
0 & 0 & t 
\end{smallmatrix}\right)$ &
$\left(\begin{smallmatrix}
1 & 0 & 0 \\
0 & t & 0 \\
-t^{-1} & 0 & t 
\end{smallmatrix}\right)$ &
$\left(\begin{smallmatrix}
1 & 0 & 0 \\
0 & t & 0 \\
0 & 0 & t 
\end{smallmatrix}\right)$
\end{tabular}
\end{center}
\vspace*{0.5cm}
\begin{center}
\begin{tabular}{c|c|c|c|c|c}
               & $24$ & $25$ & $26$ & $27$ & $28$ \\
\hline
& & & & & \\
$\ra_{\rm deg}$ & $C_1\ra C_2$ & $C_1\ra C_6(-1)$ & $C_1\ra C_3$ & 
$C_6(0)\ra C_4$ & $C_6(\be)\ra C_5(\be)$ \\
& & & & & \\
$g_t^{-1}$ & 
$\left(\begin{smallmatrix}
1 & 0 & 0 \\
0 & 1 & 0 \\
-1 & -t^2 & t 
\end{smallmatrix}\right)$ &
$\left(\begin{smallmatrix}
1 & 0 & 0 \\
0 & t^{-1} & 0 \\
0 & 0 & 1 
\end{smallmatrix}\right)$ &
 $\left(\begin{smallmatrix}
1 & 0 & 0 \\
0 & 1 & 0 \\
0 & -t^2 & t 
\end{smallmatrix}\right)$ &
$\left(\begin{smallmatrix}
1 & 0 & 0 \\
0 & 1 & 0 \\
-t & 0 & t^2  
\end{smallmatrix}\right)$ &
$\left(\begin{smallmatrix}
1 & 0 & 0 \\
0 & 1 & 0 \\
0 & 0 & t 
\end{smallmatrix}\right)$
\end{tabular}
\end{center}
\vspace*{0.5cm}
\begin{center}
\begin{tabular}{c|c|c|c|c|c}
               & $29$ & $30$ & $31$ & $32$ & $33$ \\
\hline
& & & & & \\
$\ra_{\rm deg}$ & $C_6(\be)\ra C_7(\be)$ & $C_2\ra C_7(-1)$ & $C_3\ra C_5(-1)$ & 
$C_4\ra C_5(0)$ & $D_1\ra D_2(-1)$ \\
& & & & & \\
$g_t^{-1}$ & 
$\left(\begin{smallmatrix}
1 & 0 & 0 \\
0 & 1 & 0 \\
\be & 0 & t 
\end{smallmatrix}\right)$ &
$\left(\begin{smallmatrix}
1 & 0 & 0 \\
0 & t^{-1} & 0 \\
0 & 0 & 1 
\end{smallmatrix}\right)$ &
$\left(\begin{smallmatrix}
1 & 0 & 0 \\
0 & t^{-1} & 0 \\
0 & 0 & 1 
\end{smallmatrix}\right)$ &
$\left(\begin{smallmatrix}
1 & 0 & 0 \\
0 & 1 & 0 \\
0 & 0 & t^{-1} 
\end{smallmatrix}\right)$ &
$\left(\begin{smallmatrix}
1 & 0 & 0 \\
0 & t^{-1} & 0 \\
0 & 0 & t^{-1} 
\end{smallmatrix}\right)$
\end{tabular}
\end{center}
\vspace*{0.5cm}
\begin{center}
\begin{tabular}{c|c|c|c|c|c}
               & $34$ & $35$ & $36$ & $37$ & $38$ \\
\hline
& & & & & \\
$\ra_{\rm deg}$ & $E_5(-\frac{1}{2})\ra E_6$ & $E_3\ra E_6$ & $E_3\ra E_{2,2}$ 
& $E_4\ra E_{2,\frac{1}{2}}$ & $E_4\ra E_5(-1)$ \\
& & & & & \\
$g_t^{-1}$ & 
$\left(\begin{smallmatrix}
1 & 0 & 0 \\
-2t^{-2} & -t^{-1} & 0 \\
-t^{-1} & 0 & 1 
\end{smallmatrix}\right)$ &
$\left(\begin{smallmatrix}
1 & 0 & 0 \\
0 & t^{-1} & 0 \\
0 & 0 & t^{-1} 
\end{smallmatrix}\right)$ &
$\left(\begin{smallmatrix}
2 & 0 & 0 \\
0 & 0 & t^{-1} \\
0 & 4 & 0 
\end{smallmatrix}\right)$ &
$\left(\begin{smallmatrix}
1 & 0 & 0 \\
0 & 1 & 0 \\
0 & 0 & t 
\end{smallmatrix}\right)$ &
$\left(\begin{smallmatrix}
1 & 0 & 0 \\
0 & t^{-2} & 0 \\
-1 & 0 & t^{-1} 
\end{smallmatrix}\right)$
\end{tabular}
\end{center}
\vspace*{0.5cm}
\begin{center}
\begin{tabular}{c|c|c|c|c}
               & $39$ & $40$ & $41$ & $42$ \\
\hline
& & & &  \\
$\ra_{\rm deg}$ & $E_5(\be)\ra E_{1,\frac{1}{2}}(\be)$ & $E_6\ra E_{1,2}(-1)$ 
& $E_{2,\la}\ra E_{1,\la}(-1)$ & $C_6(\be)\ra_{\rm deg}B_4(\al)$ \\
& & & & $\begin{smallmatrix}
\be\neq -1,\, \be\neq -\al, & \al\neq 0 \\
\end{smallmatrix}$ \\
$g_t^{-1}$ & 
$\left(\begin{smallmatrix}
1 & 0 & 0 \\
0 & t^{-1} & 0 \\
0 & 0 & 1 
\end{smallmatrix}\right)$ &
$\left(\begin{smallmatrix}
2 & 0 & 0 \\
0 & 0 & 1 \\
0 & t & 0 
\end{smallmatrix}\right)$ &
$\left(\begin{smallmatrix}
1 & 0 & 0 \\
0 & t^{-1} & 0 \\
0 & 0 & 1 
\end{smallmatrix}\right)$ &
$\left(\begin{smallmatrix}
0 & t & 0\\
(\be+1)t & 1 & 0 \\
\al(\al+\be)t^2 & (\al+\be)t & \al(\al+\be)t^3  
\end{smallmatrix}\right)$ 
\end{tabular}
\end{center}
\vspace*{0.5cm}
\begin{center}
\begin{tabular}{c|c|c|c}
               & $43$ & $44$ & $45$ \\
\hline
& & & \\
$\ra_{\rm deg}$ & $C_6(\be)\ra B_4(0)$ & $C_6(-\al)\ra B_4(\al)$ 
& $C_1\ra B_4(\al)$ \\
& $\begin{smallmatrix}
\be\neq -1  \\
\end{smallmatrix}$ & $\begin{smallmatrix}
\al\neq 0,1 \\
\end{smallmatrix}$ & $\begin{smallmatrix}
\al\neq 0,1 \\
\end{smallmatrix}$ \\
$g_t^{-1}$ & 
$\left(\begin{smallmatrix}
0 & t & 0\\
\frac{-1}{t(t+1)} & \frac{-1}{(\be+1)(t+1)t^2} & 1 \\
(t-\be)t^3 & (\be-t)t & 0  
\end{smallmatrix}\right)$ &
$\left(\begin{smallmatrix}
\frac{t^2}{t^2-\al} & \frac{t}{\al} & 0\\
\frac{(2\al-1)t}{t^2-\al} & 1 & \frac{(1-\al)t^2}{\al(t^2-\al)} \\
0 & \frac{t^3}{\al-t^2} & 0  
\end{smallmatrix}\right)$ &
$\left(\begin{smallmatrix}
0 & t & 0\\
t^2 & 0 & 0 \\
\al(\al-1)t^2 & (\al-1)t & \al(\al-1)t^3  
\end{smallmatrix}\right)$ 
\end{tabular}
\end{center}
\vspace*{0.5cm}
\begin{center}
\begin{tabular}{c|c|c|c|c}
               & $46$ & $47$ & $48$ & $49$ \\
\hline
& & & &  \\
$\ra_{\rm deg}$ & $C_4\ra B_4(0)$ & $C_5(\al)\ra B_4(-\al)$ 
& $C_7(\ga)\ra B_4(0)$ & $C_3\ra B_4(1)$ \\
& & $\begin{smallmatrix}
\al\neq 0,-1 \\
\end{smallmatrix}$ & $\begin{smallmatrix}
\ga\neq 0,-1 \\
\end{smallmatrix}$ & \\
$g_t^{-1}$ & 
$\left(\begin{smallmatrix}
0 & t & 0\\
-t^{-1} & -t^{-2} & 1 \\
t^2 & 0 & 0  
\end{smallmatrix}\right)$ &
$\left(\begin{smallmatrix}
0 & t & 0\\
1 & \frac{1}{(\al+1)t} & 0 \\
\frac{1}{t} & -\frac{1}{\al t^2} & 1  
\end{smallmatrix}\right)$ &
$\left(\begin{smallmatrix}
0 & t & 0\\
1 & \frac{1}{(\ga+1)t} & 0\\
\frac{1}{t} & \frac{1}{\ga t^2} & 1  
\end{smallmatrix}\right)$ &
$\left(\begin{smallmatrix}
0 & t & 0\\
t^{2} & 0 & 0 \\
t^{-1} & t^{-2} & 1  
\end{smallmatrix}\right)$
\end{tabular}
\end{center}
\vspace*{0.5cm}
\begin{center}
\begin{tabular}{c|c|c|c|c|c}
               & $50$ & $51$ & $52$ & $53$ & $54$ \\
\hline
& & & & & \\
$\ra_{\rm deg}$ & $C_2\ra B_4(0)$ & $C_5(0)\ra B_5(0)$ & $C_5(-1)\ra B_5(0)$ & 
$C_7(-1)\ra B_5(0)$ & $B_2\ra A_{12}$ \\
& & & & & \\
$g_t^{-1}$ & 
$\left(\begin{smallmatrix}
0 & t & 0\\
t^{2} & 0 & 0 \\
t^{-1} & -t^{-2} & 1  
\end{smallmatrix}\right)$ &
$\left(\begin{smallmatrix}
0 & t & 0\\
1 & 0 & 0 \\
1 & 0 & t  
\end{smallmatrix}\right)$ &
$\left(\begin{smallmatrix}
t & 0 & 0\\
0 & t^{-1} & 0 \\
0 & -t^{-1} & 1  
\end{smallmatrix}\right)$ &
$\left(\begin{smallmatrix}
0 & t & 0\\
1 & 0 & 0 \\
1 & 0 & t  
\end{smallmatrix}\right)$ &
$\left(\begin{smallmatrix}
0 & 0 & 1\\
0 & t^{-1} & 0 \\
t^{-2} & 0 & 0  
\end{smallmatrix}\right)$
\end{tabular}
\end{center}
\vspace*{0.5cm}
\begin{center}
\begin{tabular}{c|c|c|c|c|c}
               & $55$ & $56$ & $57$ & $58$ & $59$ \\
\hline
& & & & & \\
$\ra_{\rm deg}$ & $B_1\ra_{\rm deg}A_{11}$ & $B_3\ra A_5$ & $B_5(\be)\ra A_5$ & 
$B_5(\frac{1}{2})\ra A_1$ & $C_1\ra A_6$ \\
& & & $\begin{smallmatrix}
\be\neq \frac{1}{2} \\
\end{smallmatrix}$ & & \\
$g_t^{-1}$ & 
$\left(\begin{smallmatrix}
0 & 0 & 1\\
0 & 1 & 0 \\
t^{-1} & 0 & 0  
\end{smallmatrix}\right)$ &
$\left(\begin{smallmatrix}
1 & 0 & -t^{-1}\\
0 & t^{-2} & 0 \\
0 & 0 & t^{-5}  
\end{smallmatrix}\right)$ &
$\left(\begin{smallmatrix}
1 & \frac{-1}{t^{2}} & \frac{1}{(2\be-1)t}\\
0 & 1 & 0\\
0 & 0 & t^{-3}  
\end{smallmatrix}\right)$ &
$\left(\begin{smallmatrix}
t & 0 & 0\\
0 & t & 0 \\
0 & 0 & t  
\end{smallmatrix}\right)$ &
$\left(\begin{smallmatrix}
0 & t & 0\\
t^2 & 0 & 0 \\
0 & 0 & 1  
\end{smallmatrix}\right)$
\end{tabular}
\end{center}
\vspace*{0.5cm}
\begin{center}
\begin{tabular}{c|c|c|c|c|c}
               & $60$ & $61$ & $62$ & $63$ & $64$ \\
\hline
& & & & & \\
$\ra_{\rm deg}$ & $C_6(\be)\ra A_6$ & $C_6(-1)\ra A_2$ 
& $E_3\ra B_4(0)$ & $E_3\ra_{\rm deg}B_4(\al)$ & $E_4\ra B_4(0)$ \\
& $\begin{smallmatrix}
\be\neq -1 \\
\end{smallmatrix}$ & & & $\begin{smallmatrix}
\al\neq 0 \\
\end{smallmatrix}$ \\
$g_t^{-1}$ & 
$\left(\begin{smallmatrix}
0 & t & 0\\
1 & \frac{1}{(\be+1)t} & 0  \\
0 & 0 & 1  
\end{smallmatrix}\right)$ &
$\left(\begin{smallmatrix}
t & 0 & 0\\
0 & 1 & 0 \\
0 & 0 & 1  
\end{smallmatrix}\right)$ &
$\left(\begin{smallmatrix}
0 & t & 0\\
-2t^{-1} & -4t^{-2} & 1 \\
t^2 & 0 & 0  
\end{smallmatrix}\right)$ &
$\left(\begin{smallmatrix}
0 & t & 0\\
-\frac{2t^2}{\al} & -\frac{4t}{\al} & \frac{t^3}{\al}  \\
t^2 & 0 & 0  
\end{smallmatrix}\right)$ &
$\left(\begin{smallmatrix}
0 & t & 0\\
t^2(t^2+1) & -2t & -t^3/2 \\
-t^3/2 & t^2 & 0  
\end{smallmatrix}\right)$
\end{tabular}
\end{center}
\vspace*{0.5cm}
\begin{center}
\begin{tabular}{c|c|c|c|c}
               & $65$ & $66$ & $67$ & $68$ \\
\hline
& & & &  \\
$\ra_{\rm deg}$ & $E_4\ra B_4(\al)$ & $E_5(\be)\ra B_4(\al)$ 
& $E_6\ra B_3$ & $E_6\ra B_5(0)$ \\
& 
$\begin{smallmatrix}
\al\neq 0,1,  & r=\sqrt{\frac{\al}{1-\al}}\\
\end{smallmatrix}$
& 
$\begin{smallmatrix}
\be\neq -\frac{1}{2},-1,  & \al\neq 0\\
\end{smallmatrix}$
& & \\
$g_t^{-1}$ & 
$\left(\begin{smallmatrix}
0 & t & 0 \\
\frac{t^2}{1-\al} & 0 & \frac{t^3}{2(\al-1)}\\
\frac{-rt^2}{2} & rt & 0  
\end{smallmatrix}\right)$ &
$\left(\begin{smallmatrix}
0 & t & 0 \\
\frac{-4}{\al(2\be+1)t^2} & \frac{-4(2\be+\al+1)}{\al(2\be+1)^2(\be+1)t^3} & 
\frac{2}{\al(2\be+1)t}\\
1 & \frac{2}{(2\be+1)t} & 0  
\end{smallmatrix}\right)$ &
$\left(\begin{smallmatrix}
0 & t & 0\\
-4 & 4 & 2t \\
1 & 0 & 0  
\end{smallmatrix}\right)$ &
$\left(\begin{smallmatrix}
0 & t & 0\\
\frac{-2}{t} & 0 & 1 \\
1 & 0 & 0  
\end{smallmatrix}\right)$
\end{tabular}
\end{center}
\vspace*{0.5cm}
\begin{center}
\begin{tabular}{c|c|c|c|c}
               & $69$ & $70$ & $71$ & $72$ \\
\hline
& & & &  \\
$\ra_{\rm deg}$ & $E_6\ra B_5(\be)$ & $E_{2,\la}\ra B_4(0)$ 
& $E_5(-1)\ra B_4(1)$ & $E_{1,\la}(\al)\ra B_4(0)$ \\
& 
$\begin{smallmatrix}
\be\neq 0\\
\end{smallmatrix}$
&
$\begin{smallmatrix}
\la\neq 1 \\
\end{smallmatrix}$
& & $\begin{smallmatrix}
\al\neq -1,-\la, & \la\neq 1 \\
\end{smallmatrix}$ \\
$g_t^{-1}$ & 
$\left(\begin{smallmatrix}
0 & t & 0 \\
\frac{-2}{\be} & 0 & \frac{t}{\be}\\
1 & 0 & 0  
\end{smallmatrix}\right)$ &
$\left(\begin{smallmatrix}
0 & t & 0 \\
t^2 & \frac{t}{\la-1} & 0\\
\frac{1}{(1-\la)t} & \frac{-1}{(1-\la)^2t^2} & 1  
\end{smallmatrix}\right)$ &
$\left(\begin{smallmatrix}
0 & t & 0\\
4t^{-2} & -8t^{-3} & -2t^{-1} \\
1 & -\frac{2}{t} & 0  
\end{smallmatrix}\right)$ &
$\left(\begin{smallmatrix}
0 & t & 0 \\
1 & \frac{1}{(\al+1)t} & 0\\
\frac{1}{(1-\la)t} & \frac{1}{(\al+\la)(1-\la)t^2} & 1  
\end{smallmatrix}\right)$
\end{tabular}
\end{center}
\vspace*{0.5cm}
\begin{center}
\begin{tabular}{c|c|c|c|c}
               & $73$ & $74$ & $75$ & $76$ \\
\hline
& & & &  \\
$\ra_{\rm deg}$ & $E_{1,\la}(-1)\ra B_5(0)$ & $E_3\ra A_{10}$ 
& $E_4\ra_{\rm deg}A_{10}$ & $E_5(\be)\ra_{\rm deg}A_{10}$ \\
& 
$\begin{smallmatrix}
\la\neq 1\\
\end{smallmatrix}$
&
&
$\begin{smallmatrix}
i^2=- 1 \\
\end{smallmatrix}$
&  
$\begin{smallmatrix}
\be\neq -\frac{1}{2},-1\\
\end{smallmatrix}$ \\
$g_t^{-1}$ & 
$\left(\begin{smallmatrix}
0 & t & 0 \\
1 & 0 & 0\\
\frac{1}{(\la-1)t} & 0 & -1  
\end{smallmatrix}\right)$ &
$\left(\begin{smallmatrix}
0 & 0 & t\\
t^{3} & 0 & 0 \\
0 & t^2 & 0   
\end{smallmatrix}\right)$ &
$\left(\begin{smallmatrix}
\frac{t^3}{4} & \frac{t^2}{2} & t \\
0 & -t^2 & -2t\\
0 & 0 & it  
\end{smallmatrix}\right)$ &
$\left(\begin{smallmatrix}
0 & 0 & t \\
\frac{2}{(2\be+1)t} & 0 & \frac{-4}{(2\be+1)^2(\be+1)t^3} & \\
0 & 1 & \frac{2}{(2\be+1)t}  
\end{smallmatrix}\right)$
\end{tabular}
\end{center}
\vspace*{0.5cm}
\begin{center}
\begin{tabular}{c|c|c|c|c|c}
               & $77$ & $78$ & $79$ & $80$ & $81$ \\
\hline
& & & & & \\
$\ra_{\rm deg}$ & $E_6\ra A_9$ & $E_{1,1}(\al)\ra A_5$ 
& $E_{2,1}\ra_{\rm deg}A_5$ & $E_{1,1}(-1)\ra_{\rm deg}A_1$ & $D_1\ra_{\rm deg} B_4(0)$\\
& &
$\begin{smallmatrix}
\al\neq -1\\
\end{smallmatrix}$
& & & \\
$g_t^{-1}$ & 
$\left(\begin{smallmatrix}
0 & 0 & t\\
t & 0 & 0 \\
0 & 1 & 0   
\end{smallmatrix}\right)$ &
$\left(\begin{smallmatrix}
-(\al+1)t^2 & t & t^2 \\
0 & 1 & 0 & \\
0 & 0 & 1  
\end{smallmatrix}\right)$ &
$\left(\begin{smallmatrix}
-t^2 & t & 0\\
0 & t & 0 \\
0 & 0 & 1   
\end{smallmatrix}\right)$ &
$\left(\begin{smallmatrix}
t & 0 & 0\\
0 & t & 0 \\
0 & 0 & t   
\end{smallmatrix}\right)$ &
$\left(\begin{smallmatrix}
0 & t & 0\\
t^2 & 0 & 0 \\
t^2 & t & t^3   
\end{smallmatrix}\right)$
\end{tabular}
\end{center}
\vspace*{0.5cm}
\begin{center}
\begin{tabular}{c|c|c|c|c}
               & $82$ & $83$ & $84$ & $85$ \\
\hline
& & & &  \\
$\ra_{\rm deg}$ & $D_2(\al)\ra B_4(0)$ & $D_2(-1)\ra B_5(0)$ 
& $D_2(\al)\ra E_{1,1}(\al)$ & $D_2(-1)\ra E_{2,1}$ \\
& 
$\begin{smallmatrix}
\al\neq -1\\
\end{smallmatrix}$
& & & \\
$g_t^{-1}$ & 
$\left(\begin{smallmatrix}
-(\al+1)t^2 & t & -(\al+1)^2t^3 \\
t & 0 & 0 \\
0 & 1 & 0  
\end{smallmatrix}\right)$ &
$\left(\begin{smallmatrix}
t & 0 & 0\\
0 & 0 & t \\
0 & 1 & 0   
\end{smallmatrix}\right)$ &
$\left(\begin{smallmatrix}
1 & 0 & 0 \\
0 & t^{-1} & 0 \\
0 & 0 & 1  
\end{smallmatrix}\right)$ &
$\left(\begin{smallmatrix}
1 & 0 & 0 \\
-t^{-3} & t^{-1} & 0 \\
t^{-1} & -t & 1  
\end{smallmatrix}\right)$
\end{tabular}
\end{center}
% It is also useful to list some degenerations which follow from transitivity: \\[0.4cm]
%$C_5(-\frac{1}{2})\ra_{\rm deg}B_4(\frac{1}{2})\ra_{\rm deg} B_3$ \\[0.1cm]
%$C_1\ra C_2\ra B_4(0)$ \\[0.1cm]
%$C_1\ra C_3\ra B_4(1)$  \\[0.1cm]
%$C_5(\al)\ra B_5(-\al)\ra A_5$ for $\al\neq -\frac{1}{2}$, \\[0.1cm]
%$C_5( -\frac{1}{2})\ra B_4(\frac{1}{2})\ra B_3 \ra A_5$, \\[0.1cm]
%$C_7(\ga)\ra B_4(0)\ra B_5(0) \ra A_5$ for  $\ga\neq 0$, \\[0.1cm]
%$E_3\ra B_4(\frac{1}{2})\ra B_3$,  \\[0.1cm]
%$E_4\ra E_5(-1)\ra B_4(1)$, \\[0.1cm]
%$E_5(\be)\ra E_{1,\frac{1}{2}}(\be)\ra B_4(0)$ for $\be\neq -\frac{1}{2},-1$,  \\[0.1cm]
%$E_{1,\la}(-1)\ra B_5(0)\ra A_5$ for $\la\neq 1$, \\[0.1cm]
%$E_{1,\la}(\al)\ra B_4(0)\ra B_5(0)\ra A_5$ for $\al\neq -1,-\la$ and 
%$\la\neq 1$.\\[0.1cm]
%$D_2(\al)\ra B_5(0)\ra A_5$.

\end{document}